\documentclass[11pt]{amsart}  
\usepackage[left=1in,right=1in,bottom=1.2in]{geometry}  
\usepackage[T1]{fontenc}
\usepackage{amssymb,amsmath,amscd}
\usepackage{amsthm}
\usepackage[parfill]{parskip}
\usepackage{upquote}

\usepackage{hyperref}
\hypersetup{hidelinks}

\usepackage{url}

\usepackage{upgreek}
\usepackage{mathrsfs}
\usepackage{latexsym}
\usepackage{graphicx} 
\usepackage[x11names]{xcolor}
\usepackage{tikz-cd}
\usepackage{tikz} \usetikzlibrary{positioning,arrows.meta}
\usepackage{galois}
\usepackage{stmaryrd}

\usepackage{appendix}

\usepackage{enumitem}
\setlist{itemsep=0.2em, parsep=0.2em} 
\setlist[itemize,1]{label=\ensuremath{\blacktriangleright}}
\setlist[itemize,2]{label=\ensuremath{\triangleright}}

\theoremstyle{plain}
\newtheorem{theorem}{Theorem}[section]
\newtheorem{lemma}[theorem]{Lemma}

\newtheorem{criterion}[theorem]{Criterion}

\theoremstyle{definition}

\theoremstyle{remark}
\newtheorem{remark}[theorem]{Remark}

\newcommand{\animationlink}[2][Animation]{%
  \\[0.5ex]{\bf #1:} \protect\url{#2}}

\newcommand{\style}[1]{{\em #1}} 

\newcommand{\torsor}{\mathcal{P}} 

\newcommand{\coboundary}{\delta} 
\newcommand{\connecting}{\delta^*} 

\newcommand{\Z}{\mathbb{Z}} 
\newcommand{\R}{\mathbb{R}} 

\newcommand{\cgraph}{\Lambda} 
\newcommand{\coupling}{\lambda} 
\newcommand{\conf}{\mathcal{C}} 
\newcommand{\cupprod}{\smile} 

\title[Impossible by Degrees]{Impossible by Degrees: \\ Cohomology \& Bistable Visual Paradox}
\author{Lewis Ghrist}
\address{Department of Computer and Information Science \\ University of Pennsylvania \\ Philadelphia, PA 19104}
\email{lghrist@seas.upenn.edu}
\author{Robert Ghrist}
\address{Departments of Mathematics and Electrical \& Systems Engineering \\ University of Pennsylvania \\ Philadelphia, PA 19104}
\email{ghrist@math.upenn.edu}

\begin{document}

\begin{abstract}
The Penrose triangle, staircase, and related ``impossible objects'' have long been understood as related to first cohomology $H^1$: the obstruction to extending locally consistent interpretations around a loop. This paper develops a cohomological hierarchy for a class of visual paradoxes. 
Restricting to systems built from \emph{bistable} elements -- components admitting exactly two local states, such as the Necker cube's forward/backward orientations, a gear's clockwise/counterclockwise spin, or a rhombic tiling corner's convex/concave interpretation -- allows the use of $\Z_2$ coefficients throughout, reducing obstruction theory to parity arithmetic. This reveals a hierarchy of paradox classes from $H^0$ through $H^2$, refined at each degree by the relative/absolute distinction, ranging from ambiguity through impossibility to inaccessibility.
A discrete Stokes theorem emerges as the central tool: at each degree, the connecting homomorphism of relative cohomology promotes boundary data to interior obstruction, providing the uniform mechanism by which paradoxes ascend the hierarchy.

Three paradigmatic systems -- Necker cube fields, gear meshes, and rhombic tilings -- are studied in detail. Throughout, we pair cohomology with imagery and animation. To illuminate the underlying structure, we introduce the \emph{Method of Monodromic Apertures}, an animation technique that reveals monodromy through a configuration space of local sections.
\end{abstract}


\maketitle

\section{Introduction}
\label{sec:intro}

The Penrose triangle and staircase, introduced by L. Penrose and R. Penrose in \cite{penrose1958impossible}, are the archetypal {\em impossible objects}: locally coherent images admitting no globally consistent three-dimensional interpretation. In 1992, R. Penrose recast these figures in the language of \v{C}ech cohomology, classifying the triangle by the cohomology group $H^1$ of an annulus with coefficients in the multiplicative group $\R^+$ of depths \cite{Penrose1992Cohomology}. That paper contained a second figure, more enigmatic than the triangle: see Figure~\ref{fig:penrose-necker}[left].
In this image, multiple copies of the Schr\"oder staircase \cite{schroder1858optische,Donaldson2017SchroederStairs} -- an illusion of ambiguous ascent and descent (Figure~\ref{fig:penrose-necker}[center]) -- ring an annular domain with heptagonal boundary. Penrose declared this figure to have the ambiguity group of the Necker cube (Figure~\ref{fig:penrose-necker}[right]), the bistable wireframe discovered in 1832 \cite{necker1832observations}. After stating that the paradox of the heptagonal staircase is captured by $H^1$ of the annulus with coefficients in $\Z_2$, he concluded with a cryptic hint:

\begin{figure}[htb]
    \centering
    \setlength{\fboxsep}{0pt}
   \fbox{\includegraphics[height=4.75cm]{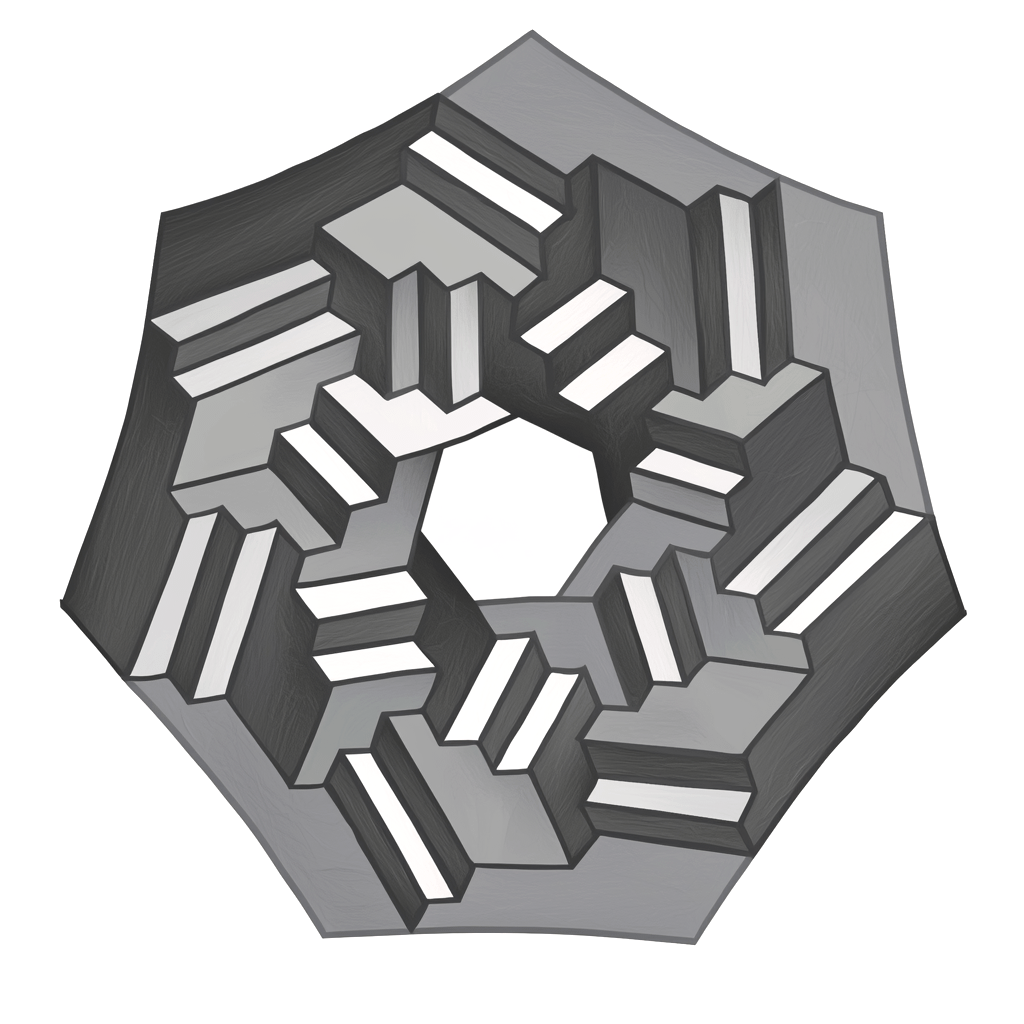}}
    \hspace{0.5em}
   \fbox{\includegraphics[height=4.75cm]{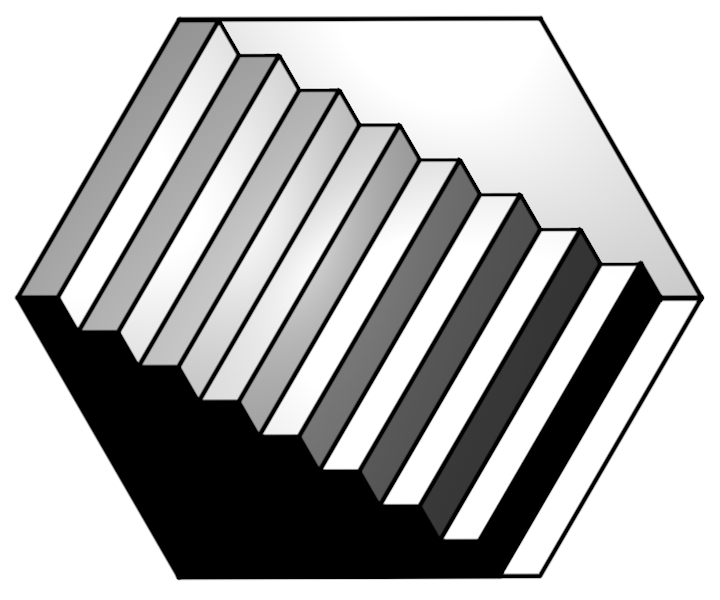}}
    \hspace{0.5em}
   \fbox{\includegraphics[height=4.75cm]{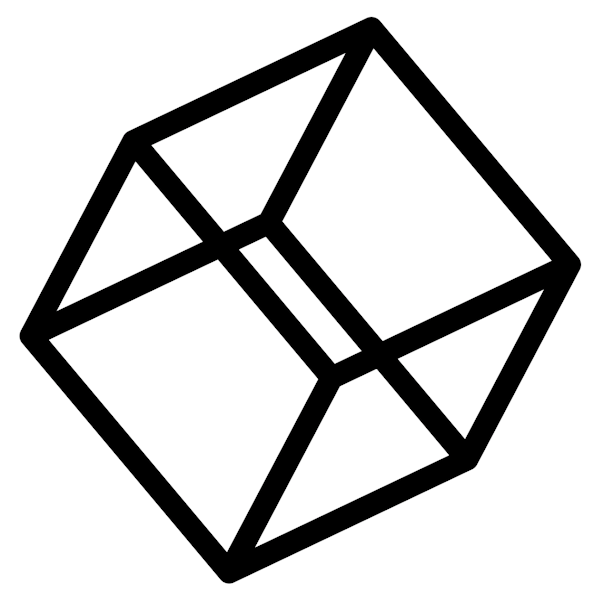}}
    \caption{\small An unusual figure [left] of Penrose from \cite{Penrose1992Cohomology} is based on the Schr\"oder staircase [center] and has the bistable ambiguity of the classical Necker cube [right].}
    \label{fig:penrose-necker}
\end{figure}

\begin{quote}
{\em ``I believe that considerations such as these may open up intriguing possibilities for further exotic types of impossible figure.''} 
\end{quote}

This paper takes up Penrose's challenge in a simplified setting, motivated by his heptagonal stairs. We develop a cohomological hierarchy for visual paradoxes built from \style{bistable} elements -- ambiguous components admitting exactly two local states. The restriction to bistable systems permits $\Z_2$ coefficients throughout, simplifying the algebra and revealing a clean taxonomy:
\begin{quote}
Base $H^0$ detects \emph{ambiguity}: multiple valid global readings.\\
Relative $H^1$ detects \emph{conflict}: boundary-forced incompatibility.\\
Absolute $H^1$ detects \emph{impossibility}: no valid global reading.\\
Relative $H^2$ detects \emph{curvature}: localized defects forced by boundary holonomy.\\
Absolute $H^2$ detects \emph{inaccessibility}: sectors unreachable by local moves.
\end{quote}
The transitions are driven by a single mechanism: the connecting homomorphism $\connecting$ of relative cohomology, which coincides with the discrete Stokes theorem.

The framework builds on recent work classifying visual paradoxes via \style{network torsors} \cite{GhristCooperband2025Obstructions}, which posited torsors as the mathematical structure underlying impossible figures, classified paradoxes via $H^1$ with general structure groups (including nonabelian examples), and analyzed boundary-induced paradoxes via relative cohomology. That paper established the framework on which this paper builds. By restricting to bistable elements, we gain both simplicity -- the machinery becomes elementary -- and a clear path upward through the cohomological hierarchy. The systematic ascent from $H^0$ through relative and absolute $H^1$ toward relative and absolute $H^2$ is the storyline. Specific contributions include: the identification of the first visual paradoxes requiring $H^2$ for their detection and classification; the discrete Stokes theorem as the uniform mechanism promoting paradoxes from one level to the next; and a novel animation technique that converts holonomy into visible monodromy by constructing torsors over configuration space.

This circle of ideas connects to sheaf-cohomological approaches to quantum contextuality \cite{AbramskyBrandenburger2011Sheaf,AbramskyBarbosa2015Paradox,Caru2017Cohomology}, where measurements depend on context and the obstruction to consistent value assignments is likewise captured by cohomology. The parallel is not accidental: both settings involve local data that cannot be globalized, and both are naturally expressed in the language of cohomology, torsors, and obstructions.

There is a ready-made theory for higher-dimensional obstructions. Just as $\Z_2$-torsors (double covers, principal $\Z_2$-bundles) are classified by $H^1$, so \emph{gerbes} provide the next categorical level, classified by $H^2$ or $H^3$ depending on the structure \cite{Brylinski1993LoopSpaces,Moerdijk2002StacksGerbes}. The framework extends to discrete and lattice settings \cite{BullivantEtAl2017HigherLattices}. However, for bistable systems the full gerbe machinery is unnecessary. Working with $\Z_2$ coefficients throughout, the classification of obstructions reduces to parity counts, and $H^2$ phenomena become accessible through elementary means: relative cohomology, cup products, and degree-shifting.

At each degree, the relative/absolute distinction separates boundary-forced from intrinsic phenomena: a \style{relative} obstruction arises from boundary data that fails to extend inward; an \style{absolute} obstruction is intrinsic to the constraint structure. The connecting homomorphism $\connecting$ of relative cohomology promotes boundary data at degree $k$ to interior obstructions at degree $k+1$ -- and the resulting computation is precisely the discrete Stokes theorem:
\begin{quote}
\emph{Boundary holonomy becomes interior curvature.}
\end{quote}
This single mechanism connects all five levels: incompatible $H^0$ boundary data yields relative $H^1$ conflict; cycles with nontrivial $H^1$ holonomy yield relative $H^2$ curvature. Between the relative and absolute at each degree, the distinction is whether the obstruction requires external intervention or is self-generated.

To make torsors visible rather than merely computed, we introduce the \emph{Method of Monodromic Apertures} (MoMA): an animation technique that converts holonomy into perceptible monodromy. A sliding window reveals local sections; after traversing a loop, the displayed section returns flipped. This is the torsor structure made manifest through motion.

The paper proceeds in three movements. The first (\S\S\ref{sec:three-systems}--\ref{sec:framework}) introduces the three paradigmatic bistable systems and formalizes their common structure via constraint graphs, coupling cochains, and the discrete Stokes theorem. The second (\S\S\ref{sec:hierarchy}--\ref{sec:MoMA}) develops the hierarchy through $H^1$: boundary-forced conflict (relative $H^1$), intrinsic impossibility (absolute $H^1$), the torsor perspective, and the Method of Monodromic Apertures for visualizing monodromy through animation. The third (\S\S\ref{sec:relH2}--\ref{sec:absH2}) ascends to $H^2$: boundary-forced curvature defects (relative $H^2$), then seam interference via cup products and degree-shifted flux models (absolute $H^2$).

While the illustrations in this paper seek to capture the nature of cohomologically-graded visual paradox, at some point still figures are insufficient, and animation is the clearer medium. Rather than try to embed low-quality animations as an Appendix or supplemental files, we have chosen to publish a public playlist of animations with commentary that illustrate the various systems in this paper (and more). The reader is encouraged to view these at: 

\begin{center}    
\protect\url{https://www.youtube.com/playlist?list=PL8erL0pXF3JaYUQHC5G0iFM0lXcGlexch}
\end{center}

\section{Three Bistable Systems}
\label{sec:three-systems}

Three paradigmatic systems exhibit bistable elements interacting through pairwise constraints: Necker cube fields, gear meshes, and rhombic tilings. We introduce each in turn, emphasizing the common structure that will unify them in \S\ref{sec:framework}.

\subsection{Necker cube fields}
\label{sec:necker}

The Necker cube is the archetypal bistable image: a wireframe cube admitting two stable depth interpretations, between which a viewer can flip at will.
Arrange multiple Necker cubes in a grid and the interpretations become coupled. Choose any cube to face forward or backward; this choice constrains its neighbors, propagating through the connected field. The entire grid inherits exactly two global states from the bistability of a single element.
Unlike the Penrose triangle, there is no inherent paradox in a Necker cube grid. But suppose we \emph{pin} certain cubes to opposite orientations, as in Figure~\ref{fig:necker-field-basic}. The agreement constraints propagating from each pinned cube conflict in the interior. This is a \style{relative paradox}: the obstruction arises not from the constraint structure itself but from incompatible boundary conditions -- \style{conflict} in the terminology of \S\ref{sec:hierarchy}. We formalize this via relative cohomology in \S\ref{sec:relH^1}.

\begin{figure}[htb]
    \centering
    \setlength{\fboxsep}{0pt}
    {\includegraphics[width=5.5in]{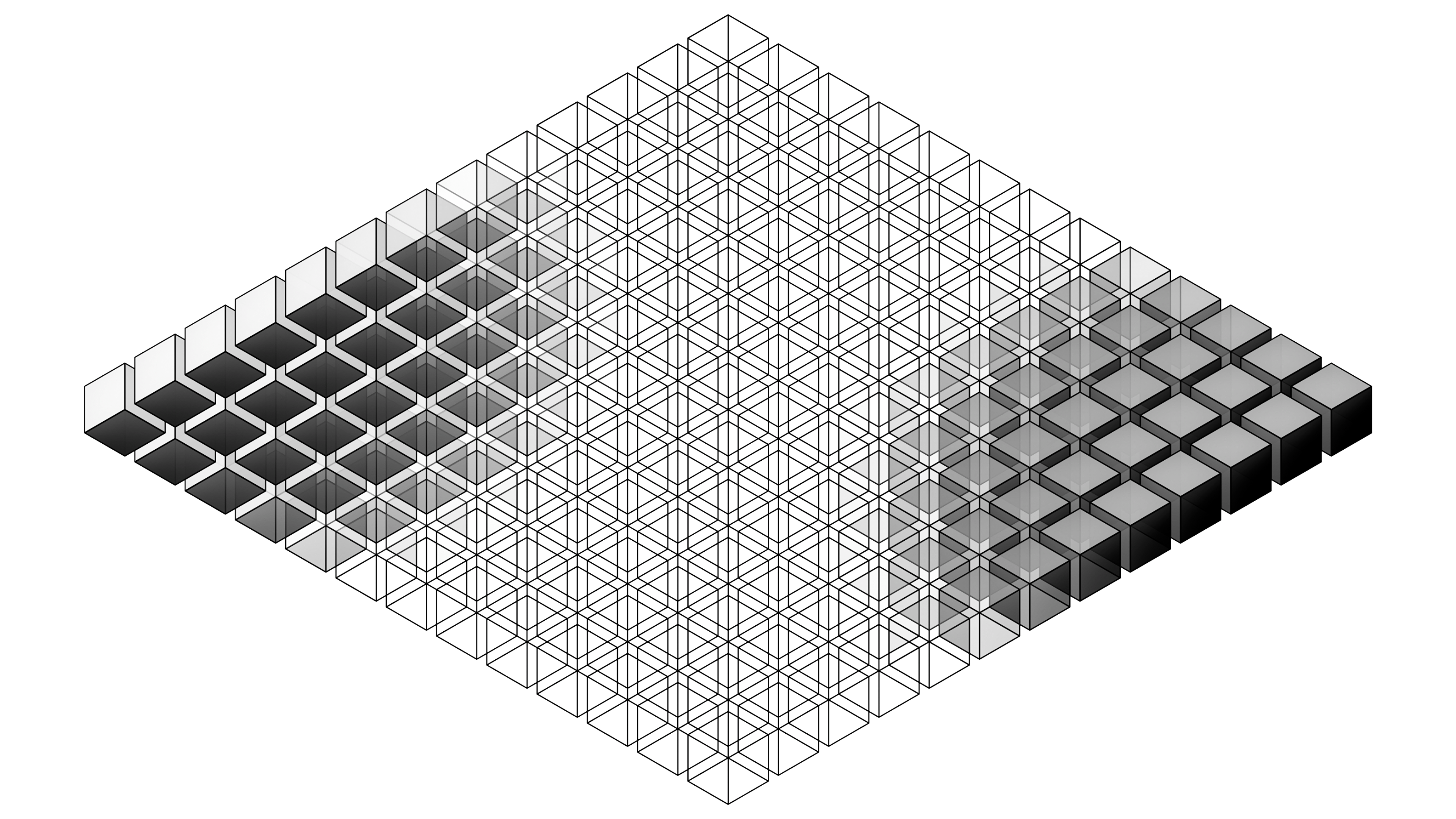}}
    \caption{\small A 2-D field of Necker cubes. Choosing an orientation for any one cube determines the orientation of all cubes in the connected component. Pinning two disjoint subsets of cubes to opposite states creates a conflict -- a relative paradox.}
    \label{fig:necker-field-basic}
\end{figure}

The static Necker cube's bistability has a dynamic counterpart. Set the wireframe rotating about a vertical axis (Figure~\ref{fig:necker-spinning}). Which way does it turn? The front/back ambiguity becomes rotational ambiguity: the perceived spin direction reverses when the viewer flips depth interpretation. This is the mechanism behind the ``spinning dancer'' illusion, and it applies to any wireframe or silhouette with front/back symmetry.

\begin{figure}[htb]
    \centering
    \setlength{\fboxsep}{0pt}
    \fbox{\includegraphics[width=6.5in]{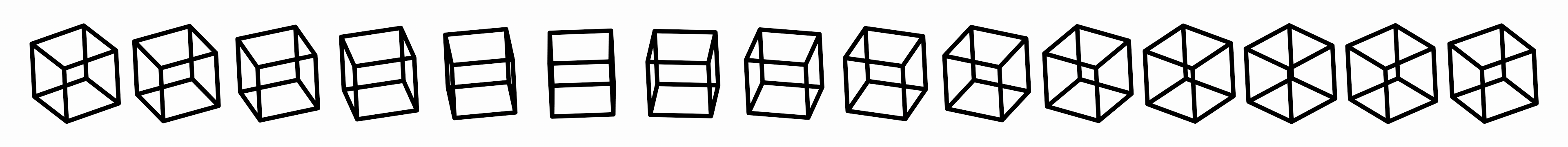}}
    \caption{\small A Necker cube rotating about a vertical axis. The depth ambiguity of the static figure becomes rotational ambiguity: the same animation can be perceived as clockwise or counterclockwise rotation.}
    \label{fig:necker-spinning}
\end{figure}

Rotational bistability converts a static classification problem -- which way does this cube face? -- into a dynamical one -- which way does it spin? -- while preserving the $\Z_2$ structure. Animation is not merely illustration but a tool for revealing paradox, as we shall see in \S\ref{sec:MoMA}.

\subsection{Gear meshes}
\label{sec:gears}

Rotational bistability finds natural mechanical expression in gear systems. When two external (spur) gears mesh, their teeth interlock and force \emph{opposite} rotations: if one turns clockwise, its neighbor must turn counterclockwise. This \style{opposition constraint} transforms isolated bistability into coupled variables. Choose a spin direction for any gear, and the mesh determines all others -- provided no contradiction arises.

Consider $n$ external gears in a ring, each meshing with its two neighbors (Figure~\ref{fig:gear-ring}). If $n$ is even, alternating spins around the loop satisfies all opposition constraints, and the system spins freely. If $n$ is odd, no such alternation exists -- the final gear must simultaneously agree and oppose its neighbor. The system \emph{locks}: a mechanical paradox that we will identify as a nontrivial $H^1$ class.

\begin{figure}[htb]
    \centering
    \includegraphics[width=5in]{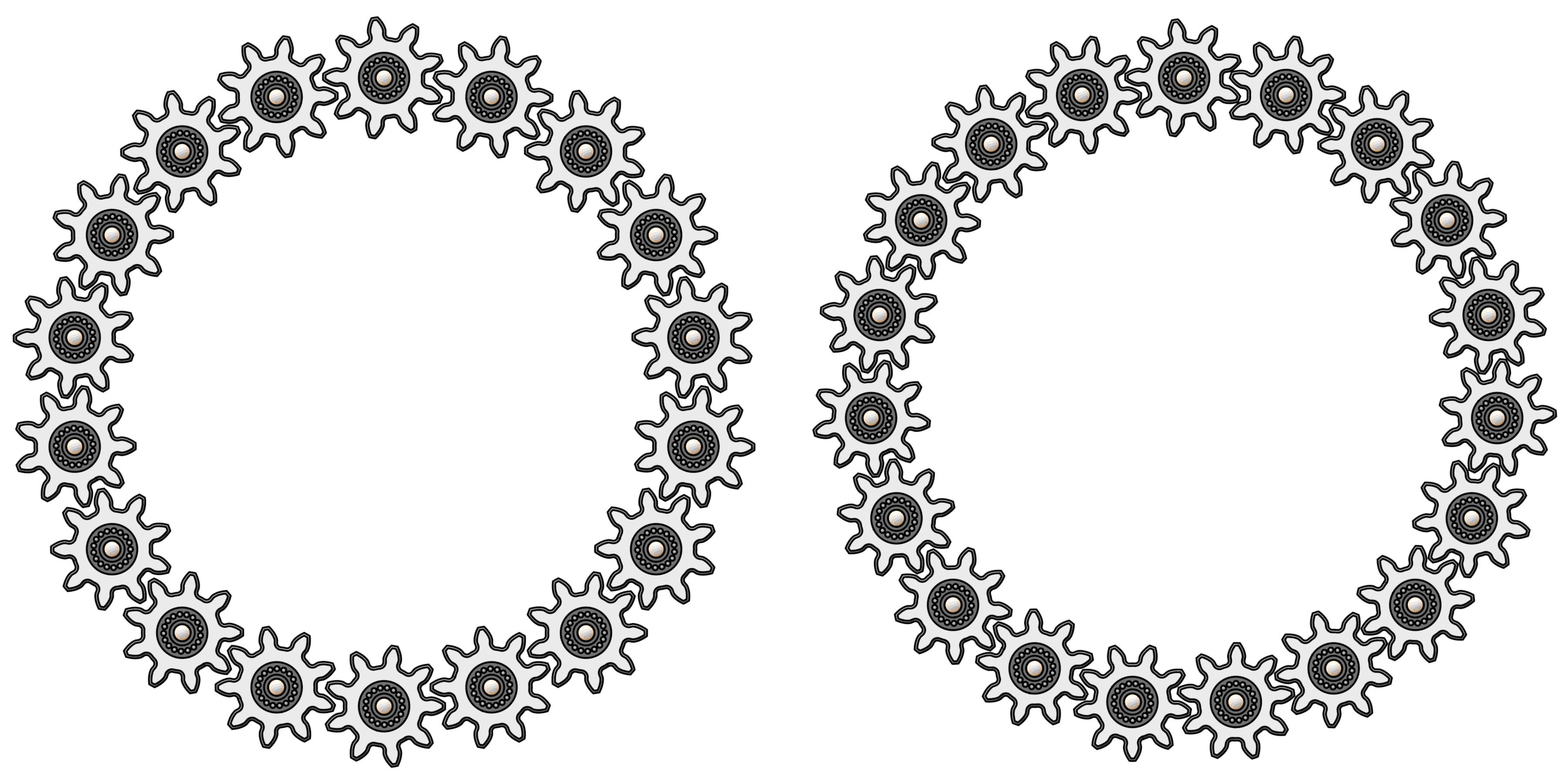}
    \caption{\small With an even number of gears in a ring, the system spins freely [left], 
     but an odd number of gears locks [right].}
    \label{fig:gear-ring}
\end{figure}

Planar gear meshes extend this to two dimensions. A square grid of gears, each meshing with four neighbors, has a bipartite contact graph: vertices partition into two classes (checkerboard coloring), with all edges between classes. Assign clockwise to one class, counterclockwise to the other, and all opposition constraints are satisfied. More generally, a gear mesh with pure opposition is consistent if and only if its contact graph contains no odd cycles.
Ring gears and belts introduce agreement constraints within otherwise opposition-dominated systems; bevel gears extend the picture to three dimensions. We restrict attention to external spur gears in this paper.

\subsection{Rhombic tilings}
\label{sec:stepped}

Rhombic tilings of the plane -- tessellations by parallelograms with equal side lengths -- admit natural interpretations as projections of three-dimensional stepped surfaces. The regular \style{lozenge tiling}, built from $60^\circ$ rhombi, arises as the projection along the $(1,1,1)$ direction of a monotone surface in $\Z^3$: a ``pile of cubes'' whose boundary facets project to the rhombi: see Figure~\ref{fig:tilings}[left].

The bistable elements are \style{degree-3 vertices}: corners where exactly three rhombi meet. At such a vertex, the incident rhombi suggest a cube corner seen in projection -- either \style{convex} (pointing toward the viewer) or \style{concave} (pointing away). This binary choice is the bistable variable.
When two degree-3 vertices are opposite corners of the same rhombus, they are \style{linked}. A consistent stepped-surface interpretation requires linked vertices to take opposite states: one convex, the other concave. This is the same opposition constraint we encountered with meshing gears, now encoded in tiling geometry.
The lozenge tiling's constraint graph -- degree-3 vertices connected by linkages across rhombus diagonals -- forms a hexagonal lattice. This graph is bipartite: all cycles have even length, and consistent alternation of convex/concave is always possible. No obstruction arises, just as an even gear cycle spins freely.

\begin{figure}[htb]
    \centering
    \setlength{\fboxsep}{0pt}
  \fbox{\includegraphics[height=6cm]{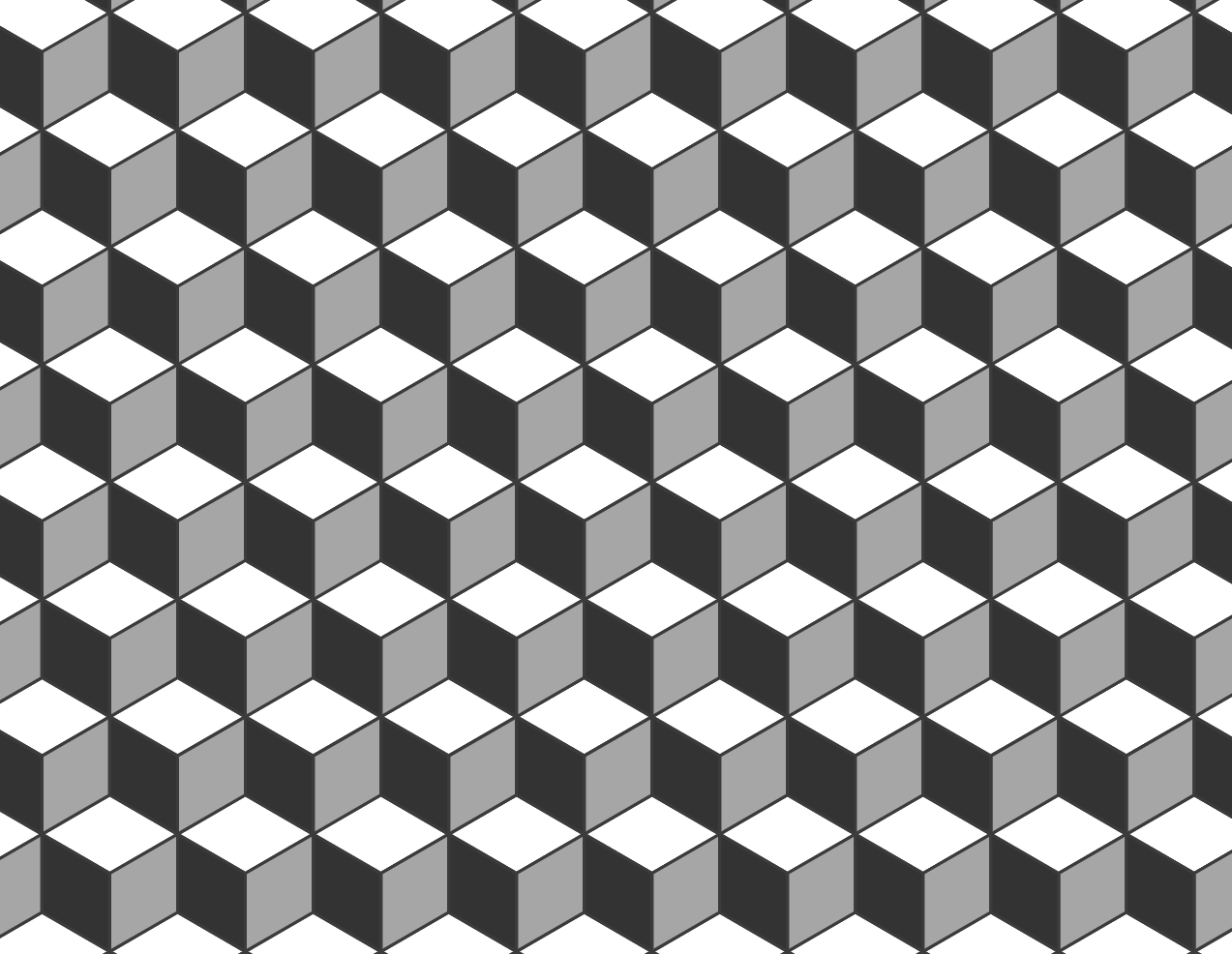}}
   \hspace{0.5em}
  \fbox{\includegraphics[height=6cm]{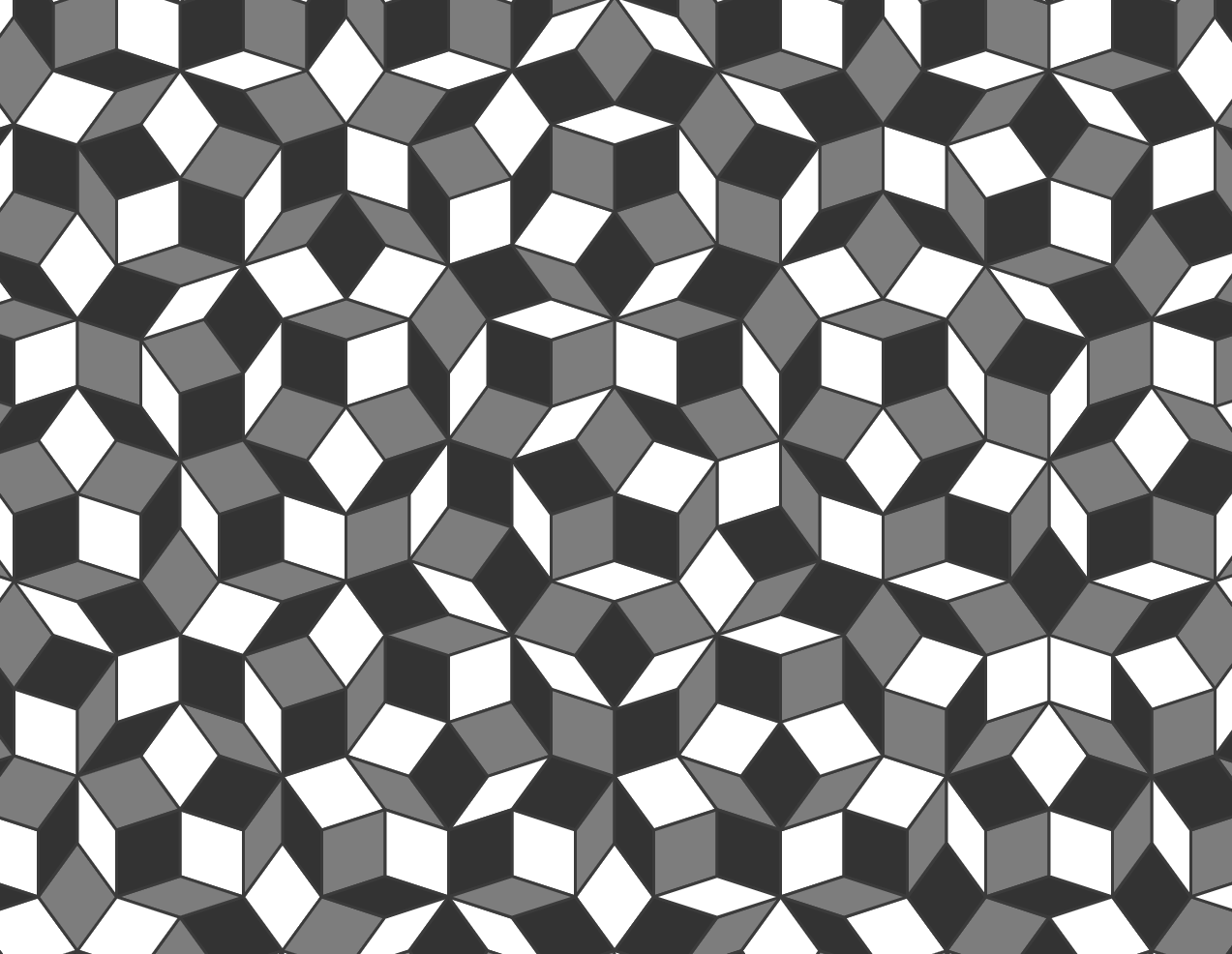}}
    \caption{\small Two rhombic tilings: [left] the regular lozenge tiling; [right] a patch of Penrose P3 tiling, centered on a five-fold rosette.}
    \label{fig:tilings}
\end{figure}

The \style{Penrose P3 tiling} is different. Built from two rhombic prototiles with angles $72^\circ/108^\circ$ and $36^\circ/144^\circ$, it exhibits quasiperiodic order with local five-fold symmetry. Its characteristic feature is the \style{pentagonal rosette}: five thick rhombi arranged around a central vertex, as in Figure~\ref{fig:tilings}[right]. The rosette contains five degree-3 corners forming a 5-cycle in the constraint graph, each edge carrying an opposition constraint. Attempting to alternate convex and concave around the pentagon fails for the same reason five meshing gears lock: the cycle length is odd, and no consistent assignment exists.

This is precisely the structure of Penrose's enigmatic heptagonal staircase (Figure~\ref{fig:penrose-necker}[left]), with cycle length 5 rather than 7. Both configurations consist of bistable elements in an odd ring with opposition constraints. The Schr\"oder stairs play exactly the role of degree-3 corners: each admits two interpretations (ascending versus descending), and adjacent elements must differ for local consistency.

Longer odd cycles -- of length 15, 35, and beyond -- thread through the P3 tiling, each presenting the same obstruction. The parallel with gear meshes is exact: degree-3 vertices play the role of gears, linkages play the role of meshing contacts, and the opposition constraint (convex versus concave) plays the role of counter-rotation. 

\section{The Mathematical Framework}
\label{sec:framework}

The three bistable systems of the previous section share a common structure that we now make precise. Throughout, we use the following correspondences:
\begin{quote}
A \style{state} $x \in C^0(\cgraph; \Z_2)$ assigns a binary percept to each vertex: forward/backward, clockwise/counterclockwise, or convex/concave.

A \style{coupling} $\coupling \in C^1(\cgraph; \Z_2)$ encodes pairwise constraints: $\coupling(e) = 0$ demands agreement; $\coupling(e) = 1$ demands opposition.

A \style{global section} is a state satisfying all constraints: $\coboundary x = \coupling$.

The \style{obstruction} to existence is the cohomology class $[\coupling] \in H^1(\cgraph; \Z_2)$.
\end{quote}
The key players are a \style{constraint graph} $\cgraph$ encoding which elements interact, the \style{coupling cochain} $\coupling$ specifying how, and a \style{curvature} $\mu = \coboundary\coupling$ measuring local frustration when 2-cells are present. The central result is Stokes' theorem, relating boundary holonomy to interior curvature.

\subsection{Constraints and couplings}
\label{sec:coupling}

The constraint graph $\cgraph$ has bistable elements as vertices and edges connecting elements whose states are coupled. In a \emph{Necker cube field}, each cube is a vertex; edges connect cubes sharing a face (agreement constraint). In a \emph{gear mesh}, each gear is a vertex; edges connect meshing gears (opposition constraint). In a \emph{rhombic tiling}, each degree-3 corner is a vertex; edges connect opposite corners of the same rhombus (opposition constraint). Note that in the rhombic case, the edges of $\cgraph$ are \emph{not} edges of the original tiling -- they are diagonals connecting degree-3 corners across each rhombus. The tiling serves as scaffolding; once $\cgraph$ is extracted, the original edges disappear from the analysis. Figure~\ref{fig:constraint-graphs} illustrates this for the P3 tiling.

\begin{figure}[htb]
    \centering
    \setlength{\fboxsep}{0pt}
  \fbox{\includegraphics[height=6cm]{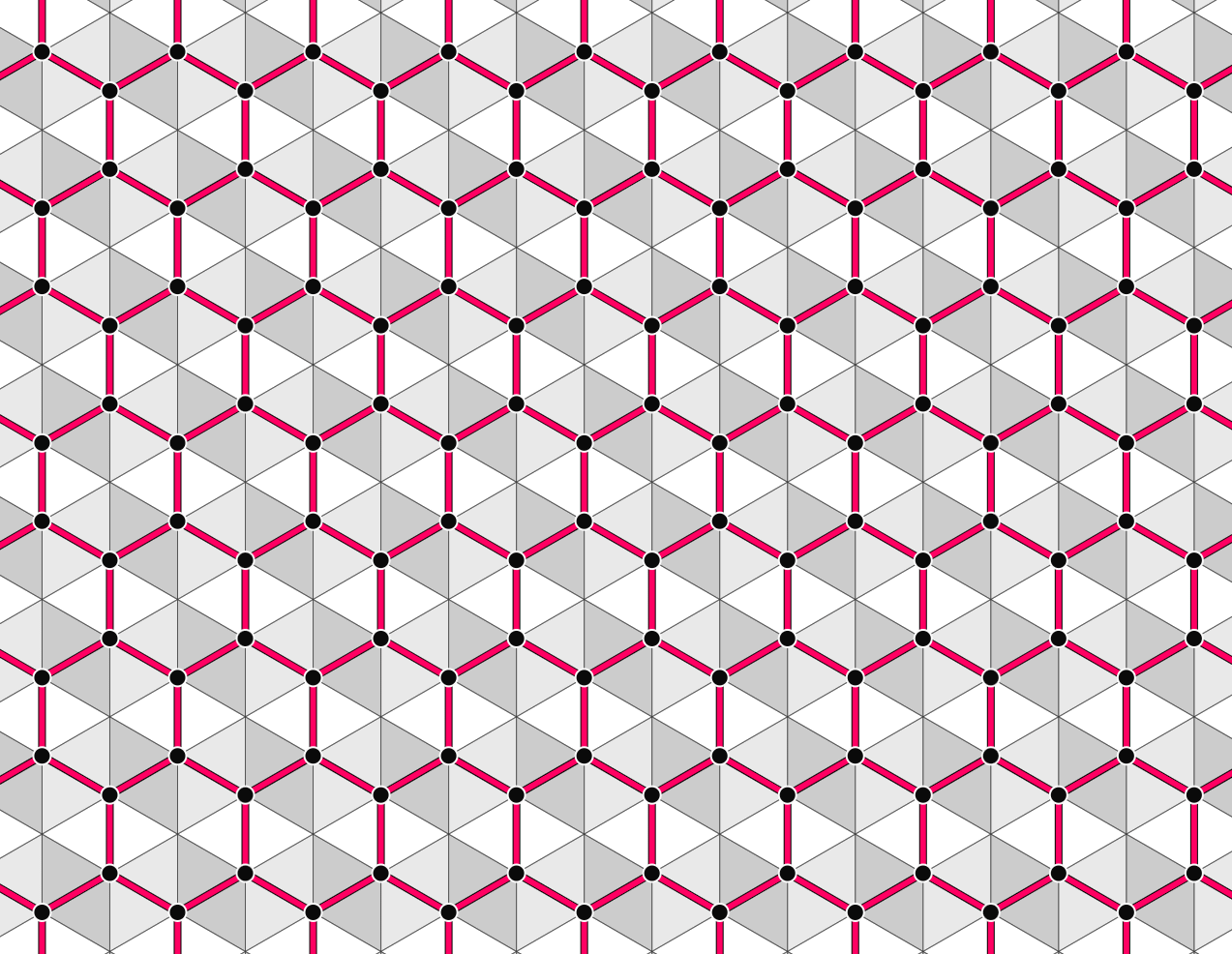}}
   \hspace{0.5em}
  \fbox{\includegraphics[height=6cm]{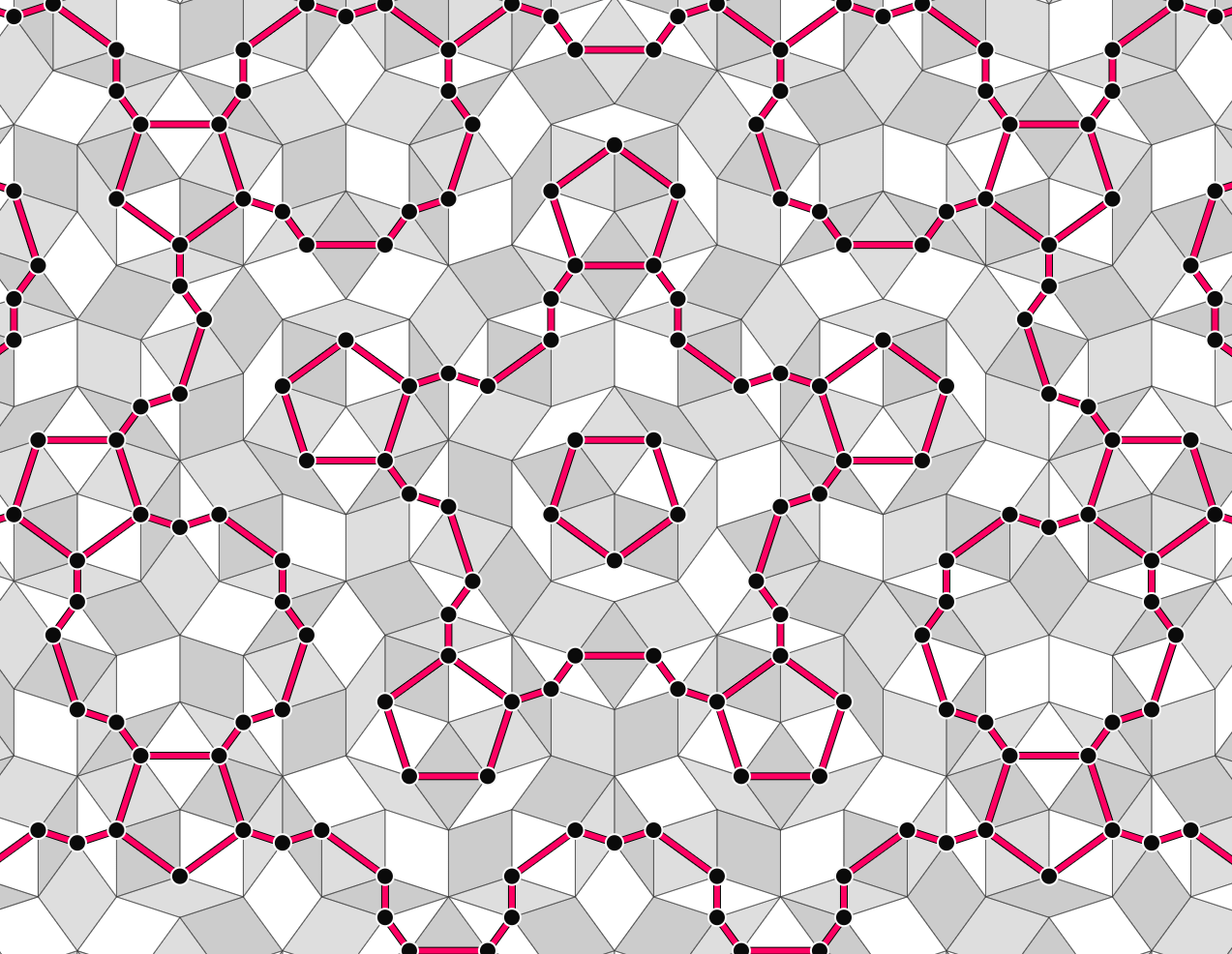}}
    \caption{\small Constraint graphs $\cgraph$ extracted from the two rhombic tilings of Figure~\ref{fig:tilings}. Vertices are degree-3 corners; red edges connect linked pairs across rhombus diagonals. The lozenge tiling yields a hexagonal graph with only even cycles [left], while the P3 tiling yields pentagons at each rosette, plus longer cycles, even and odd [right].}
    \label{fig:constraint-graphs}
\end{figure}

The coupling cochain $\coupling \in C^1(\cgraph; \Z_2)$ specifies constraint types: $\coupling(e) = 0$ means the endpoints of $e$ must \emph{agree}; $\coupling(e) = 1$ means they must \emph{oppose}.
For Necker fields, $\coupling \equiv 0$ (agreement); for gear meshes and rhombic tilings, $\coupling \equiv 1$ (opposition).

A state $x \in C^0(\cgraph; \Z_2)$ is \style{globally consistent} if $\coboundary x = \coupling$, where the coboundary $(\coboundary x)(e) = x(u) + x(v)$ computes the parity of states across each edge. Whether solutions exist is the question of whether $\coupling$ is a coboundary.

For some systems, the constraint graph $\cgraph$ embeds naturally in an ambient cell complex $X$: a grid of gears in the plane, a field of Necker cubes in space. We take $\cgraph \subset X^{(1)}$ as a subgraph of the 1-skeleton. Edges of $X^{(1)}$ not in $\cgraph$ -- the \style{free edges} -- carry no bistable constraint.

For rhombic tilings, the constraint edges are diagonals, not tiling edges, so no natural ambient complex presents itself. We construct $X$ by adding free edges until the augmented graph bounds polygonal 2-cells. In Figure~\ref{fig:P3-cell-structure}, red edges form $\cgraph$; blue edges complete the cell structure.

\begin{figure}[htb]
    \centering
    \setlength{\fboxsep}{0pt}
   \fbox{\includegraphics[width=5.75in]{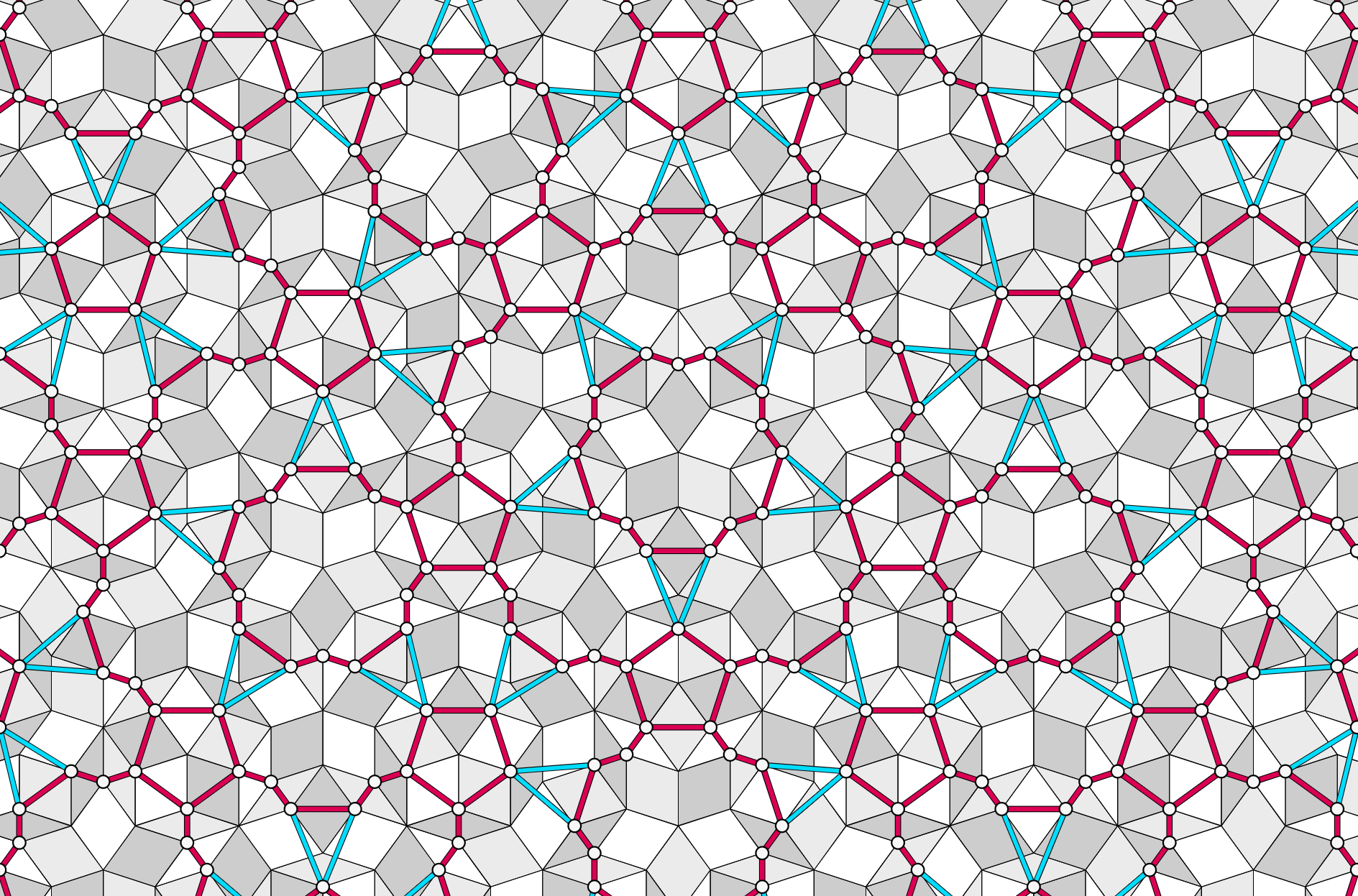}}
    \caption{\small A patch of P3 tiling with its associated cell complex $X$. Red edges form the constraint graph $\cgraph$, connecting degree-3 corners across rhombus diagonals. Blue edges are added (based on degree four vertices) to complete a regular 2-dimensional cell structure. The 2-cells bounded by the red and blue edges fill the plane.}
    \label{fig:P3-cell-structure}
\end{figure}

The result is a cell complex $X$ with $\cgraph \subset X^{(1)}$. An \style{extension} $\tilde\coupling \in C^1(X; \Z_2)$ agrees with $\coupling$ on $\cgraph$ and may be chosen arbitrarily on free edges. Different extensions yield different curvatures, addressed below.

\subsection{The holonomy criterion}
\label{sec:holonomy}

The equation $\coboundary x = \coupling$ is solvable if and only if $[\coupling] = 0$ in $H^1(\cgraph; \Z_2)$. Appendix~\ref{sec:appendix-cohomology} provides background; the essential point is a parity count around cycles.

\begin{criterion}[Holonomy Criterion]
\label{crit:holonomy}
Let $\cgraph$ be a connected graph and $\coupling \in C^1(\cgraph; \Z_2)$ a 
coupling cochain. A global state $x$ with $\coboundary x = \coupling$ exists 
if and only if every cycle $\gamma$ in $\cgraph$ satisfies 
$\sum_{e \in \gamma} \coupling(e) = 0$.
\end{criterion}

\begin{proof}
If $\coupling = \coboundary x$, then for any cycle $\gamma = (v_0, v_1, \ldots, v_n = v_0)$,
\[
\sum_{e \in \gamma} \coupling(e) = \sum_{i=0}^{n-1} \bigl(x(v_{i+1}) + x(v_i)\bigr) = 0,
\]
since each vertex appears exactly twice. Conversely, suppose $\coupling$ evaluates 
to zero on every cycle. Fix a basepoint $v_0$ and define $x(v)$ as the sum 
$\sum_{e \in p} \coupling(e)$ along any path $p$ from $v_0$ to $v$. This is 
well-defined: two paths from $v_0$ to $v$ differ by a cycle, and $\coupling$ 
vanishes on cycles. By construction, $(\coboundary x)(e) = x(v) + x(u) = \coupling(e)$ 
for each edge $e = uv$.
\end{proof}

The sum $\sum_{e \in \gamma} \coupling(e)$ is the \style{holonomy} of $\coupling$ around $\gamma$. Nonzero holonomy certifies \style{impossibility} -- an \style{absolute} paradox intrinsic to the constraints. When holonomy vanishes on all cycles, global states exist but may still be obstructed by incompatible boundary data -- a \style{relative} paradox, experienced as \style{conflict}, detected by relative cohomology.

\subsection{Curvature and Stokes' theorem}
\label{sec:curvature-stokes}

When $X$ is 2-dimensional, extending $\coupling$ to $X^{(1)}$ and applying the coboundary yields the \style{curvature} $\mu = \coboundary\tilde\coupling \in C^2(X; \Z_2)$. If $\cgraph = X^{(1)}$, no extension is needed; otherwise, different extensions yield different curvature representatives.

For a face $f$ with boundary edges $e_1, \ldots, e_k$, the curvature is $\mu(f) = \sum_i \tilde\coupling(e_i)$. A face is \style{flat} if $\mu(f) = 0$ and \style{frustrated} if $\mu(f) = 1$. Changing the extension on a free edge toggles $\mu$ on the two incident faces, redistributing curvature. The intrinsic quantity is identified by:

\begin{lemma}[Extension-independence]
\label{lem:extension-independence}
Let $D \subset X$ be a 2-chain whose boundary $\partial D$ lies entirely in the 
constraint graph $\cgraph$. Then the total curvature $\sum_{f \subset D} \mu(f)$ 
depends only on $\coupling|_{\partial D}$, not on the choice of extension 
$\tilde{\coupling}$.
\end{lemma}

\begin{proof}
By definition of coboundary, $\sum_{f \subset D} \mu(f) = \sum_{f \subset D} (\coboundary\tilde{\coupling})(f) = \sum_{e \in \partial D} \tilde{\coupling}(e)$, the last equality following from the fact that each interior edge of $D$ appears in exactly two face boundaries (canceling in $\Z_2$), while each boundary edge appears in exactly one. Since $\partial D \subset \cgraph$, every edge in the boundary sum lies in $\cgraph$, where $\tilde{\coupling}$ agrees with $\coupling$ by definition.
\end{proof}

This lemma justifies speaking of ``the curvature inside $D$'' without specifying 
an extension, provided $\partial D \subset \cgraph$. Henceforth we invoke this 
convention tacitly.

Cochains and chains pair by evaluation: $\langle c, \tau \rangle = \sum_i c(\sigma_i)$ for a $k$-cochain $c$ and $k$-chain $\tau = \sum_i \sigma_i$.

\begin{theorem}[Stokes]
\label{thm:stokes}
For any $k$-cochain $c \in C^k(X; \Z_2)$ and any $(k+1)$-chain $\tau \in C_{k+1}(X; \Z_2)$,
\[
\int_\tau \coboundary c \;=\; \int_{\partial\tau} c .
\]
\end{theorem}

Applied to curvature $\mu = \coboundary\tilde\coupling$ and a 2-chain $\tau$, Stokes asserts:
\[
\sum_{f \subset \tau} \mu(f) \;=\; \sum_{e \in \partial\tau} \tilde\coupling(e).
\]
When $\partial\tau \subset \cgraph$, the right-hand side depends only on $\coupling|_\cgraph$: interior curvature is determined by boundary holonomy. On a closed surface, total curvature vanishes.

This is the paper's central principle: boundary inconsistency cannot dissolve -- it must localize as curvature somewhere. The $H^2$ constructions of later sections all manifest this as boundary cycles with nonzero holonomy forcing interior frustration.

\begin{remark}[Gauge perspective]
\label{rem:gauge}
The framework admits a gauge-theoretic interpretation. The coupling $\coupling$ is a discrete connection with curvature $\mu = \coboundary\coupling$ as field strength. A gauge transformation $x \mapsto x + \xi$ replaces $\coupling$ by the cohomologous $\coupling + \coboundary\xi$, preserving the cohomology class $[\coupling]$ and leaving $\mu$ invariant. This perspective becomes essential in \S\ref{sec:flux}, where bistable variables shift from vertices to edges: $\coupling$ becomes a potential, $\mu$ becomes prescribed flux, and the obstruction moves from $H^1$ to $H^2$.
\end{remark}

\section{A Cohomological Hierarchy}
\label{sec:hierarchy}

Cohomology does not begin at degree one. The group $H^0(\cgraph; \Z_2)$ 
controls the structure of solutions: if $x$ solves $\coboundary x = \coupling$, 
so does $x + 1$, and when solutions exist they come in pairs related by this 
global flip. For a connected constraint graph, $H^0(\cgraph; \Z_2) \cong \Z_2$ 
acts freely and transitively on the solution set, which therefore forms a 
principal homogeneous space with no distinguished element. A bistable image 
admitting consistent global readings thus has genuine indeterminacy: which 
reading is correct? Neither is privileged, and the viewer's perceptual 
system oscillates between them. This is \style{ambiguity}, the degree-zero 
phenomenon, and the Necker cube is its paradigm.

With the passage from $H^0$ to $H^1$, ambiguity gives way to obstruction, 
and the nature of the obstruction depends on its origin. A 
\style{relative} paradox arises when an otherwise consistent system is forced 
into contradiction by incompatible boundary conditions: the interior could 
accommodate either assignment, but not both at once. This is \style{conflict} -- 
a gradient of competing influences, formalized by the connecting homomorphism 
$\connecting: H^0(A) \to H^1(X, A)$, which promotes incompatible degree-zero 
boundary data to a degree-one interior obstruction. The Necker interval, 
with endpoints pinned to opposite orientations, is the paradigm.

An \style{absolute} paradox requires no external intervention. When some 
cycle in the constraint graph carries nonzero holonomy, $[\coupling] \neq 0$ 
in $H^1(\cgraph)$, and no global state exists regardless of boundary 
conditions. This is \style{impossibility}, intrinsic to the constraint 
structure. The odd gear ring is the paradigm.

The mechanism by which boundary data becomes interior obstruction is 
the connecting homomorphism $\connecting$ of relative cohomology. Given 
a cocycle $\alpha$ on $A$, one extends it arbitrarily to a cochain 
$\tilde\alpha$ on $X$ and takes the coboundary: 
$\connecting[\alpha] = [\coboundary\tilde\alpha]$. But this is precisely 
what Stokes computes. For any chain $\tau$ with $\partial\tau \subset A$:
\[
\langle \coboundary\tilde\alpha, \tau \rangle 
\;=\; \langle \tilde\alpha, \partial\tau \rangle
\;=\; \langle \alpha, \partial\tau \rangle,
\]
the middle equality by definition of coboundary, the last because 
$\tilde\alpha$ agrees with $\alpha$ on $A \supset \partial\tau$. The 
interior integral on the left is determined by the boundary integral 
on the right. This is discrete Stokes, and it shows that 
$\coboundary\tilde\alpha$ depends only on $\alpha$, not on the 
extension --- exactly what makes $\connecting$ well-defined.

The connecting homomorphism $\delta^*$ {\em is} both Stokes and the obstruction: 
it measures precisely the failure of boundary data to extend inward. 
This mechanism is uniform across degrees, and we elevate it to a guiding principle:

\begin{quote}
\textbf{Stokes Principle.} \emph{Boundary inconsistency becomes interior 
obstruction. At each degree $k$, a class in $H^k(A)$ that fails to extend 
inward promotes via $\connecting$ to a class in $H^{k+1}(X, A)$, computed 
as the coboundary of any extension.} 
\end{quote}

The connecting homomorphism and discrete Stokes are two views of the same machine
operating uniformly across degrees. At $k = 0$: disjoint  
components of a Necker field pinned to opposite states fixes  
data in $H^0(A)$ (as in Figure \ref{fig:necker-field-basic}); 
the connecting homomorphism $\connecting: H^0(A) \to H^1(X, A)$ 
sends this to a relative $H^1$ obstruction -- conflict in the interior. 
At $k = 1$: a cycle $\gamma$ with nonzero holonomy bounding a disc $D$ 
represents a class in $H^1(\gamma)$; the connecting homomorphism 
$\connecting: H^1(\gamma) \to H^2(D, \gamma)$ sends this to a 
relative $H^2$ obstruction --- curvature that must localize somewhere 
inside $D$. The Stokes formula
\[
\sum_{f \subset D} \mu(f) \;=\; \sum_{e \in \gamma} \coupling(e)
\]
gives the computation directly: boundary holonomy on the right forces 
interior frustration on the left. The obstruction cannot dissolve; it 
must go somewhere.

The same mechanism operates one degree higher still, but accessing 
absolute $H^2$ classes requires moving beyond relative cohomology. Two 
routes present themselves. The \style{cup product} 
$\alpha \smile \beta \in H^2(X)$ measures interference between 
independent $H^1$ classes; when nonzero, their Poincar\'e-dual seams 
must intersect regardless of cocycle representatives. Alternatively, 
\style{degree-shifting} places bistable variables on edges rather than 
vertices, with prescribed flux playing the role of coupling; the 
obstruction to finding a potential then lives directly in $H^2$.

The full hierarchy, refined by the relative/absolute distinction at each degree, is:
\[
\begin{array}{ccccccccc}
H^0(A) & \xrightarrow{\;\;\connecting\;\;} & H^1(X,A) & \longrightarrow & H^1(X) & \xrightarrow{\;\;\connecting\;\;} & H^2(X,A) & \longrightarrow & H^2(X) \\[1.5ex]
\textit{ambiguity} && \textit{conflict} && \textit{impossibility} && \textit{curvature} && \textit{inaccessibility}
\end{array}
\]
The arrows labeled $\connecting$ are the Stokes transitions: boundary data at degree $k$ promoting to interior obstruction at degree $k+1$. At each degree, the gap separates boundary-forced (relative) from intrinsic (absolute) phenomena.\footnote{The spaces $X$ and $A$ are instantiated differently at each level: for $H^0/H^1$, $X$ is the constraint graph $\cgraph$ with boundary vertices $A$; for $H^1/H^2$, $X$ is the ambient 2-complex with $A = \partial D \subset \cgraph$.} A relative obstruction becomes absolute when the bounding cycle is not itself a boundary in the ambient space -- when the topology that sourced the boundary data is intrinsic rather than imposed.

What changes across this hierarchy is the experiential character of the obstruction. \emph{Ambiguity} ($H^0$) is the coexistence of multiple valid readings with nothing to choose between them. \emph{Conflict} (relative $H^1$) is the gradient of competing boundary forces in an otherwise consistent system -- the visual tension in the Necker interval, where cubes near each pinned end are strongly determined while central cubes fade to wireframe indeterminacy. \emph{Impossibility} (absolute $H^1$) is intrinsic locking, requiring no external intervention -- the odd gear ring that cannot turn. \emph{Curvature} (relative $H^2$) is localized frustration forced by boundary holonomy: defects that can redistribute among faces but cannot vanish, conserved by a discrete Gauss law. \emph{Inaccessibility} (absolute $H^2$) is the partition of configurations into sectors by a topological invariant -- free local motion that nonetheless cannot reach certain global states.

\section{Relative $H^1$: Conflict}
\label{sec:relH^1}

Paradox can be bootstrapped from ambiguity when incompatible boundary conditions are imposed on an otherwise consistent system. The result is \style{conflict}: a gradient of competing influences that the interior cannot resolve. The long exact sequence of a pair $(X, A)$ captures this precisely: the connecting homomorphism $\connecting: H^0(A) \to H^1(X, A)$ promotes degree-zero boundary data to a degree-one interior obstruction.

The mechanism is always the same. Take a contractible domain $X$ with trivial coupling, choose a subcomplex $A$ with two connected components, and pin distinct components to different states. Any global state restricts to the \emph{same} state on all components, so boundary data assigning different states to different components lies outside the image of restriction, yielding a nontrivial class in $H^1(X, A)$. The richness lies not in the algebra -- which is always $\Z_2$ -- but in the geometry of how the resulting conflict distributes through space.

\subsection{Necker cube fields}
\label{sec:necker-boundary}

The canonical example, introduced in \cite{GhristCooperband2025Obstructions}, is the \style{Necker interval}: a linear chain of Necker cubes with endpoints forced to opposite orientations (Figure~\ref{fig:necker-interval}). The base space $X = P_n$ is a path of $n+1$ vertices with boundary $A = \{v_0, v_n\}$, and the coupling $\coupling \equiv 0$ enforces agreement throughout. Since $X$ is contractible, $H^1(X) = 0$ and no absolute paradox exists. The long exact sequence collapses to
\[
0 \to H^0(X) \xrightarrow{\mathrm{res}} H^0(A) \xrightarrow{\connecting} 
H^1(X, A) \to 0.
\]
Here $H^0(X) \cong \Z_2$ (two uniform states) and $H^0(A) \cong \Z_2 \times \Z_2$ (four boundary assignments). A uniform state restricts to equal endpoints, so $\mathrm{im}(\mathrm{res}) = \{(0,0), (1,1)\}$, and exactness yields $H^1(X, A) \cong \Z_2$. The boundary data $\beta = (0, 1)$ -- left cube forward, right cube backward -- lies outside the image of restriction, certifying a nontrivial relative $H^1$ class.

\begin{figure}[htb]
    \centering
    \setlength{\fboxsep}{0pt}
   \fbox{\includegraphics[width=5.5in]{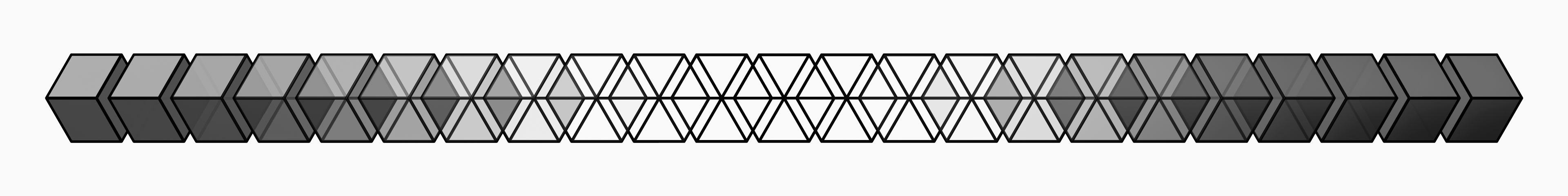}}
    \caption{\small The Necker interval: endpoints forced to opposite 
    orientations create a relative $H^1$ paradox. Cubes near the boundary 
    are strongly determined; central cubes fade to wireframe ambiguity.}
    \label{fig:necker-interval}
\end{figure}

Perceptually, conflict manifests as a gradient of ambiguity: cubes near the forced endpoints appear strongly determined, while central cubes fade to wireframe indeterminacy, caught between incompatible constraints propagating from both ends. This perceptual gradient is not predicted by the hard-constraint model -- which records only the existence of an obstruction -- but reflects a soft-constraint or energy-minimization process in visual perception. The term ``conflict'' captures this experiential character: unlike impossibility, which is uniform and immediate, conflict is distributed and graded.

\begin{figure}[htb]
    \centering
    \includegraphics[width=6.5in]{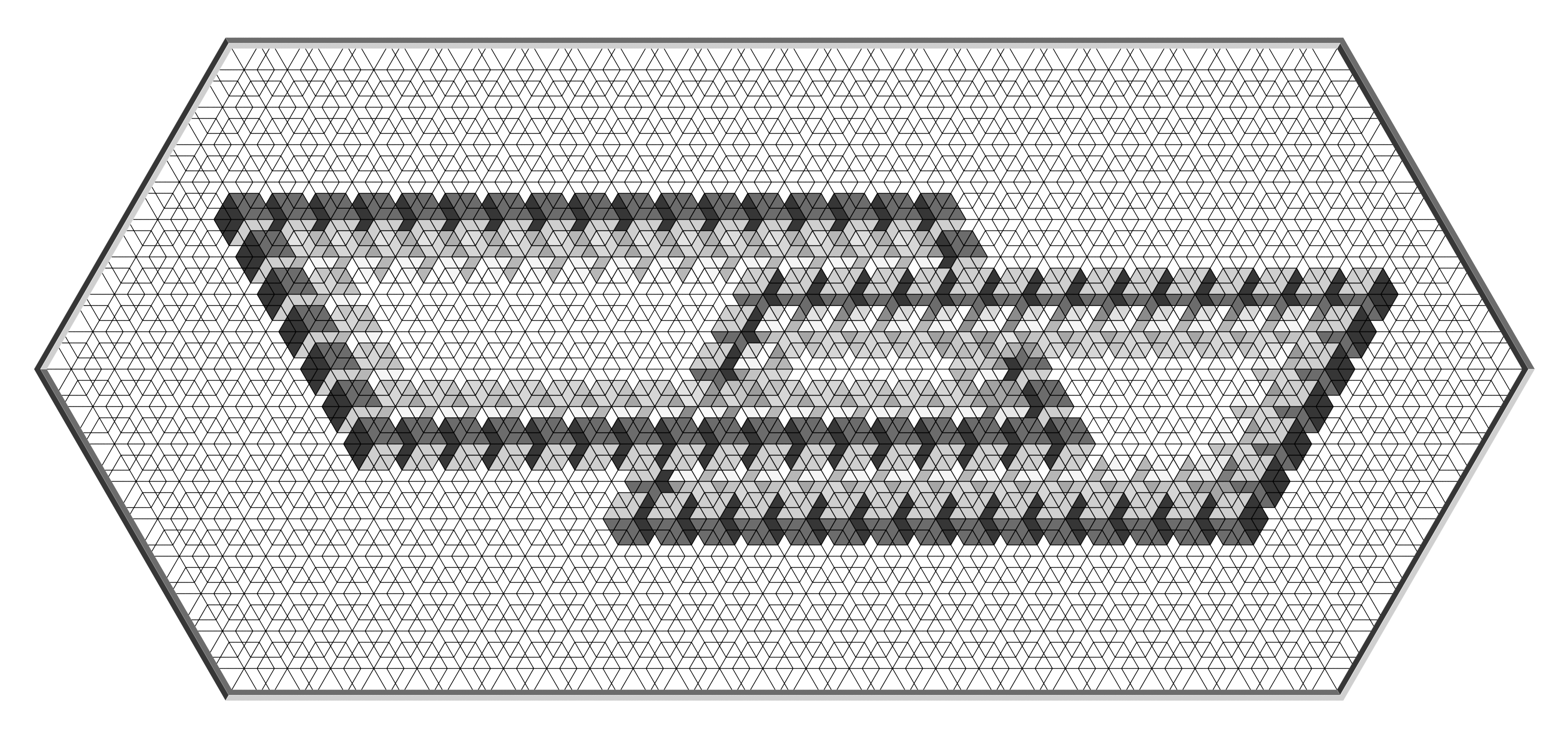}
    \hspace{1em}
    \caption{\small A three-dimensional Necker cube field with two linked loops forced to 
    opposite orientations. The conflict zone threads between the loops, its geometry shaped by the linking -- though the cohomology class $H^1(X,A) \cong \Z_2$ is insensitive to this.}
    \label{fig:necker-3d}
\end{figure}

The same computation generalizes freely. Any connected, contractible Necker field $X$ with $\coupling \equiv 0$ and any decomposition $A = A_1 \cup A_2$ into disjoint nonempty subsets produces $H^1(X, A) \cong \Z_2$: pinning $A_1$ and $A_2$ to opposite orientations yields the generator. In two dimensions, a planar grid with well-separated pinned regions creates a beautiful conflict zone, as in Figure~\ref{fig:necker-field-basic}. In three dimensions, the topology of the boundary components enriches the picture. Figure~\ref{fig:necker-3d} shows a cubical lattice with two \emph{linked} loops of cubes forced to opposite orientations. The cohomology group $H^1(X, A) \cong \Z_2$ is the same as for two unlinked loops or two isolated points -- linking is invisible at this algebraic level. The difference is geometric: how the conflict distributes through space. Cohomology detects the existence of the obstruction; finer invariants would be needed to capture how it is situated.

\subsection{Lozenge tilings with boundary forcing}
\label{sec:lozenge-boundary}

The same phenomenon appears in two-dimensional rhombic tilings with richer visual content. Consider a rectangular patch of lozenge tiling whose constraint graph $\cgraph$ -- degree-3 vertices connected across rhombus diagonals -- forms a portion of the hexagonal lattice. This graph is bipartite: vertices partition into two classes with all edges between classes. Assign convex to one class, concave to the other, and all opposition constraints are satisfied. Both checkerboard solutions exist; $[\coupling] = 0$ in $H^1(\cgraph)$, and there is no absolute paradox.

Now impose boundary conditions. Let $A = A_L \cup A_R$ consist of degree-3 vertices along the left and right edges of the patch. For a suitably chosen patch width, both $A_L$ and $A_R$ lie in the same bipartite class. Any global checkerboard solution assigns the same state to all vertices in that class, hence to both boundaries. Forcing $A_L$ convex while $A_R$ is concave yields boundary data outside the image of restriction, and the long exact sequence gives $H^1(X, A) \cong \Z_2$ with this data as generator.

The result is the Necker interval lifted to two dimensions: mathematically identical ($H^1(X,A) \cong \Z_2$), but visually a stepped-surface paradox rather than a depth-orientation one. The left boundary appears to jut toward the viewer, the right boundary recedes, and the interior fades to flat ambiguity. Figure~\ref{fig:lozenge-boundary} shows the result. 

\begin{figure}[htb]
    \centering
    \setlength{\fboxsep}{0pt}
    \fbox{\includegraphics[width=6.5in]{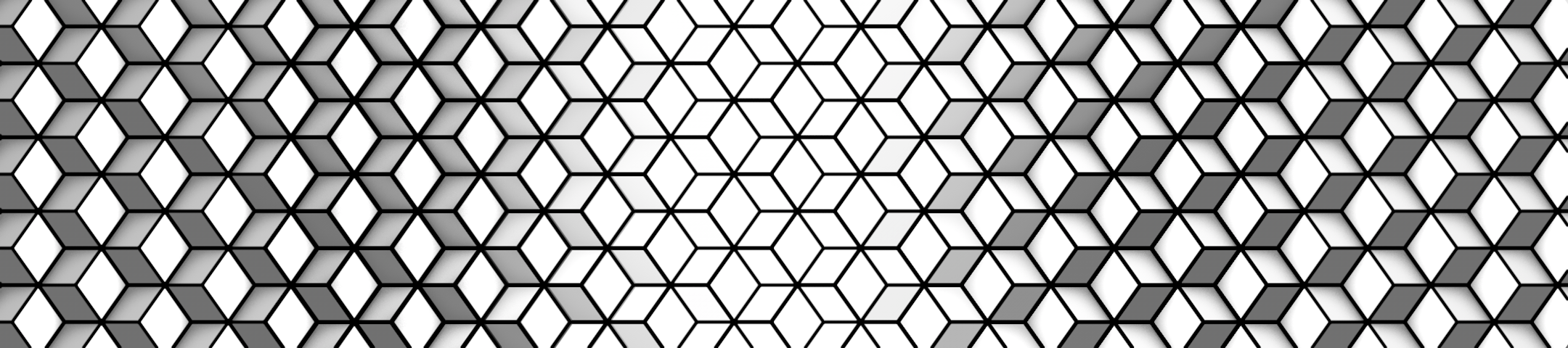}}
    \caption{\small A lozenge tiling patch with left boundary forced to one 
    orientation and right boundary forced to the 
    opposite orientation via shading; {\em cf.} Figure \ref{fig:necker-interval}.}
    \label{fig:lozenge-boundary}
\end{figure}

\section{Absolute $H^1$: Impossibility}
\label{sec:absH1}

The relative paradoxes of the preceding section arose from external intervention: pinning distant elements to incompatible states created conflict that the interior could not resolve. We now turn to \style{impossibility} -- paradoxes requiring no such intervention, where the coupling $\coupling$ itself forces nonzero holonomy around some cycle, making global states impossible regardless of boundary conditions.

The holonomy criterion (Criterion~\ref{crit:holonomy}) provides the test: $\coupling$ admits a global state if and only if every cycle $\gamma$ has zero holonomy $\sum_{e \in \gamma} \coupling(e) = 0$. When some cycle has nonzero holonomy, the class $[\coupling] \in H^1(\cgraph; \Z_2)$ is nontrivial and no solution exists. These are \style{absolute} paradoxes, intrinsic to the constraint structure.

\subsection{Odd gear cycles}
\label{sec:odd-gears}

The cleanest examples arise in gear systems with pure opposition constraints. For $n$ external gears in a ring, each meshing with its two neighbors, every edge carries $\coupling(e) = 1$ and the holonomy around the cycle is $n \bmod 2$. When $n$ is even, alternating spins satisfy all constraints; when $n$ is odd, the system locks. What we observed as kinematics in \S\ref{sec:gears} we now recognize as a nontrivial $H^1$ class.

The picture grows richer on surfaces. Consider a rectangular grid of external gears on the flat torus $T^2$, built by identifying opposite edges of the grid (Figure~\ref{fig:gear-torus}). Each gear meshes with four neighbors; the constraint graph is the standard square lattice on the torus. When horizontal and vertical dimensions are both odd, both fundamental cycles carry nonzero holonomy, representing the class $[\coupling] = \alpha + \beta \in H^1(T^2; \Z_2) \cong \Z_2 \oplus \Z_2$, where $\alpha$ and $\beta$ are generators dual to horizontal and vertical cycles respectively. These are independent obstructions: resolving one (say, by cutting along a horizontal circle) leaves the other intact. This multiplicity will become significant when we examine cup products in \S\ref{sec:cup-product}.

\begin{figure}[htb]
    \centering
    \includegraphics[width = 6in]{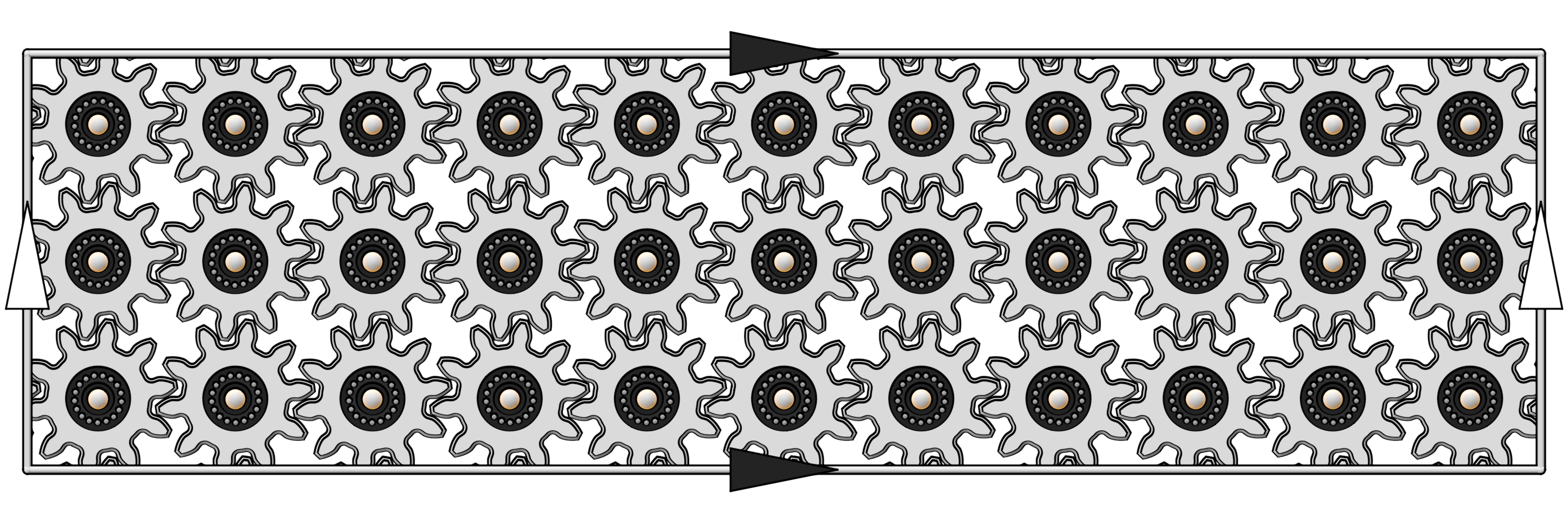}
    \caption{\small A planar gear mesh on a flat torus with odd numbers of gears along both longitude and meridian generates nonzero holonomy in both fundamental $H^1$ classes.}
    \label{fig:gear-torus}
\end{figure}

\subsection{Rhombic tilings and zonohedra}
\label{sec:tilings-zonohedra}

The P3 rosette, introduced in \S\ref{sec:stepped}, provides the canonical planar example. Five thick rhombi surround a central vertex, their degree-3 corners forming a 5-cycle in the constraint graph with $\coupling \equiv 1$ on every edge. The holonomy is $5 \equiv 1$, certifying a nontrivial $H^1$ class: no consistent assignment of convex and concave exists around the rosette. This is precisely the structure of Penrose's heptagonal staircase (Figure~\ref{fig:penrose-necker}[left]) at a different cycle length -- both are odd rings of bistable elements with opposition constraints.
Longer odd cycles present the same obstruction of $H^1$ defects: see Figure \ref{fig:P3-shading} for an example of local concave/convex assignments that cannot be continued globally.

\begin{figure}[htb]
    \centering
    \setlength{\fboxsep}{0pt}
   \fbox{\includegraphics[height=6cm]{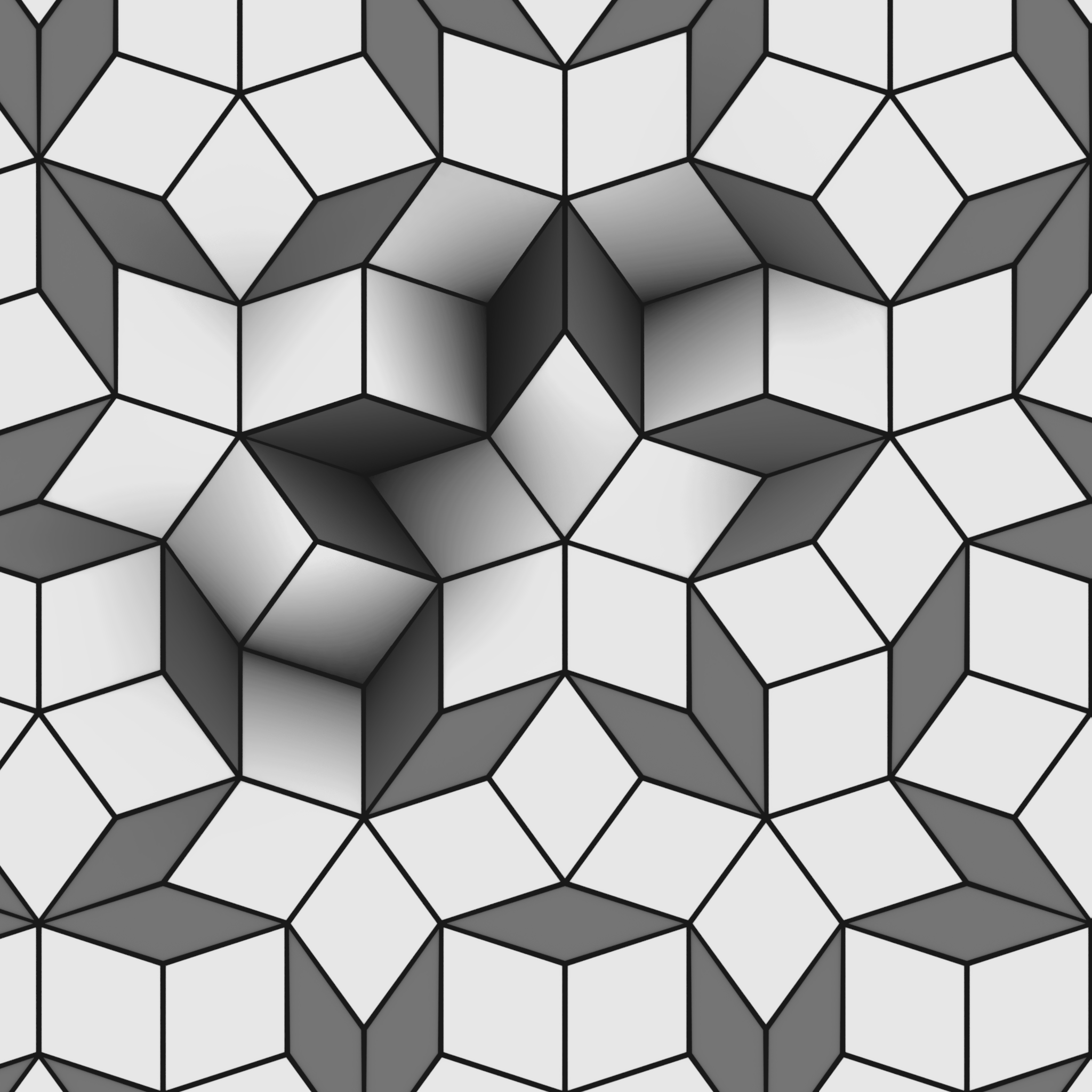}}
    \hspace{4em}
   \fbox{\includegraphics[height=6cm]{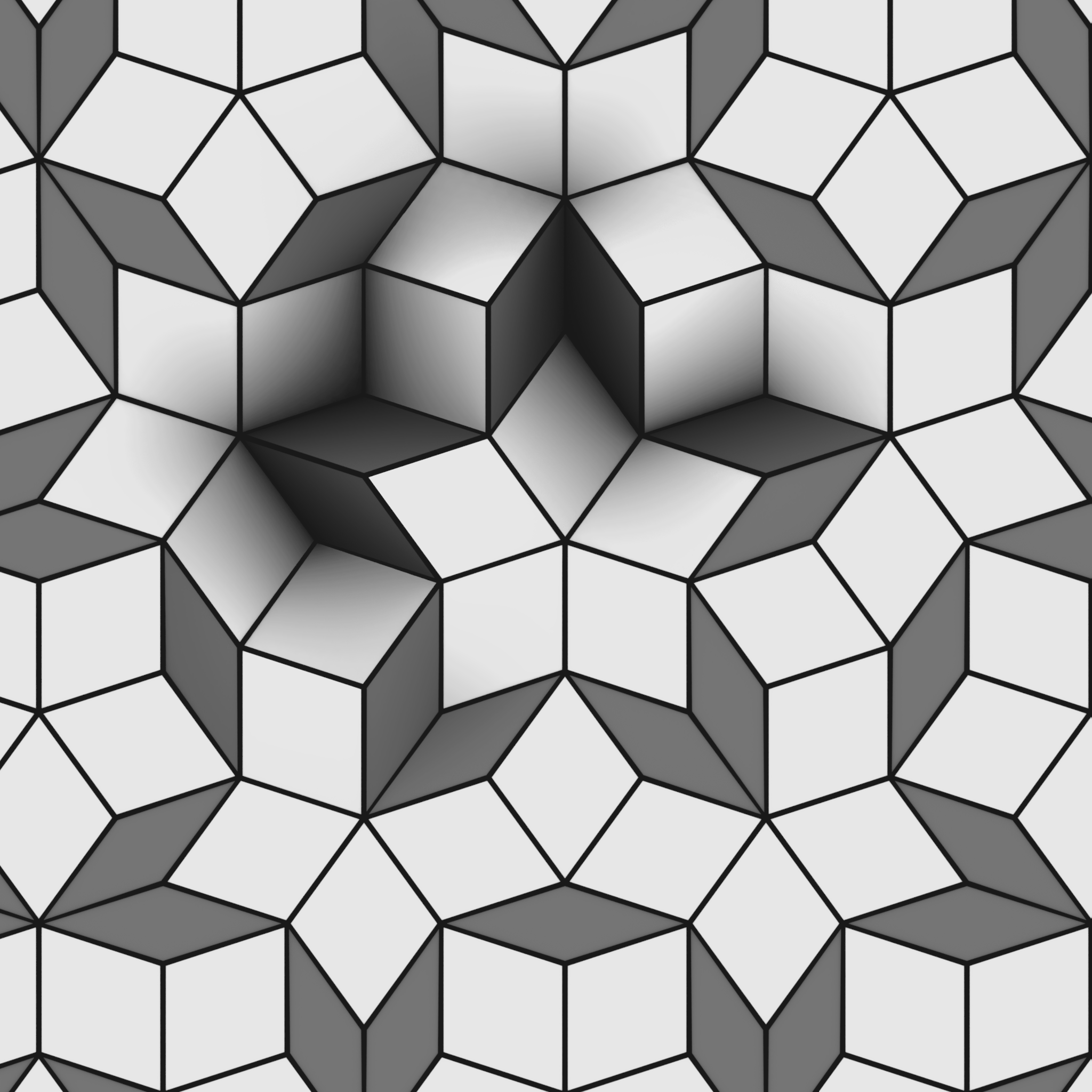}}
    \caption{\small Lifting corners in the P3 rhombic quasitiling along an odd cycle in $\cgraph$; the two lifts (indicated by shading) are locally consistent, but cannot be extended globally around the odd cycle due to the nonzero holonomy.}
    \label{fig:P3-shading}
\end{figure}

The same phenomenon transplants beautifully to closed surfaces via \style{rhombic zonohedra} -- convex polyhedra whose faces are all rhombi \cite{Coxeter1973RegularPolytopes}. The \style{rhombic triacontahedron} (RT30) has 30 golden rhombi arranged with icosahedral symmetry (Figure~\ref{fig:zonohedra}[left]). Its 20 degree-3 vertices sit at the 3-fold symmetry axes, and each of the 30 rhombic faces contributes one constraint edge -- the diagonal connecting its two degree-3 corners. The resulting constraint graph is the dodecahedral graph: 20 vertices, 30 edges, with twelve pentagonal faces surrounding the 5-fold axes. Each pentagon is a 5-cycle of opposition constraints with nonzero holonomy -- exactly the P3 rosette on the sphere.

\begin{figure}[htb]
    \centering
    \setlength{\fboxsep}{0pt}
   \fbox{\includegraphics[height=4cm]{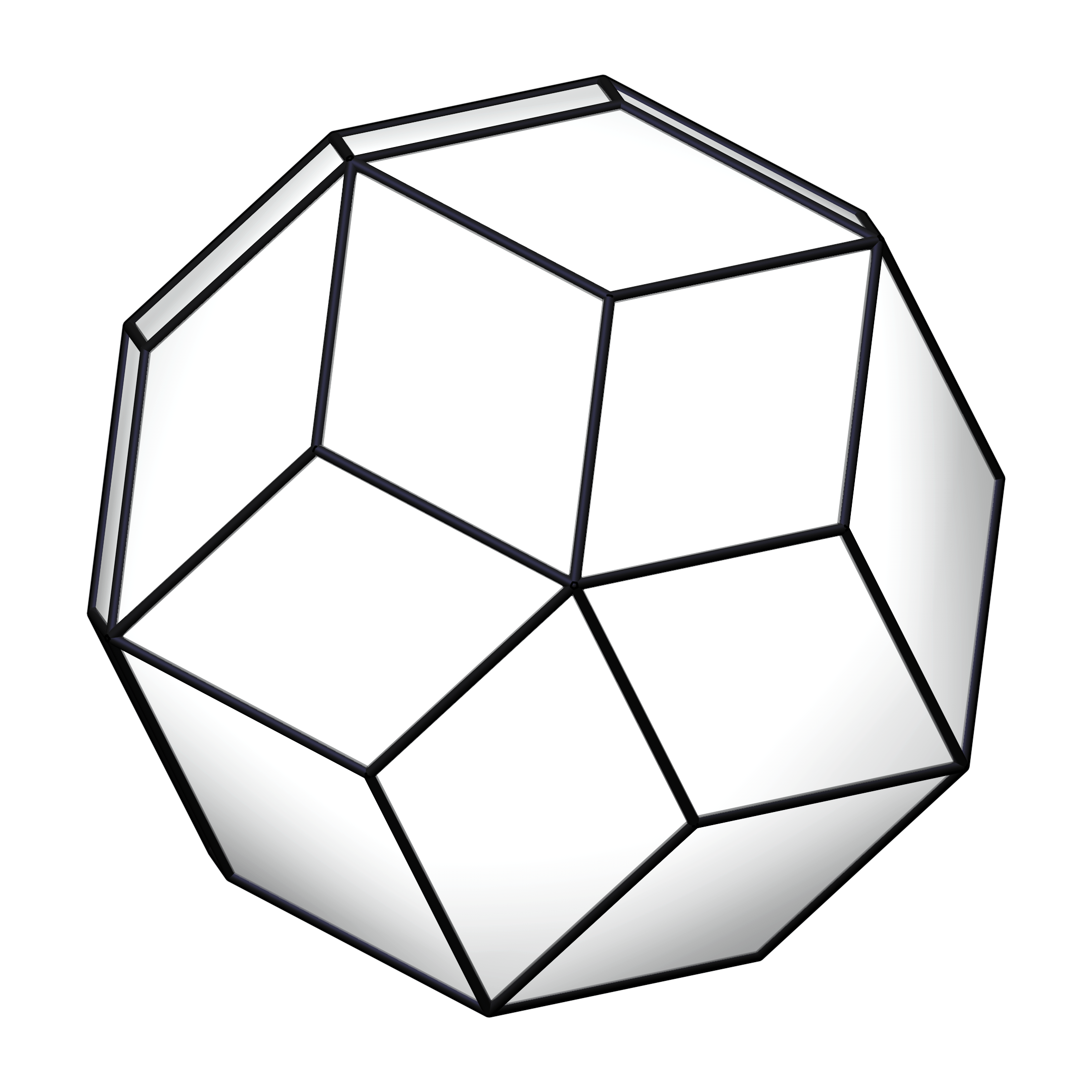}}
   \fbox{\includegraphics[height=4cm]{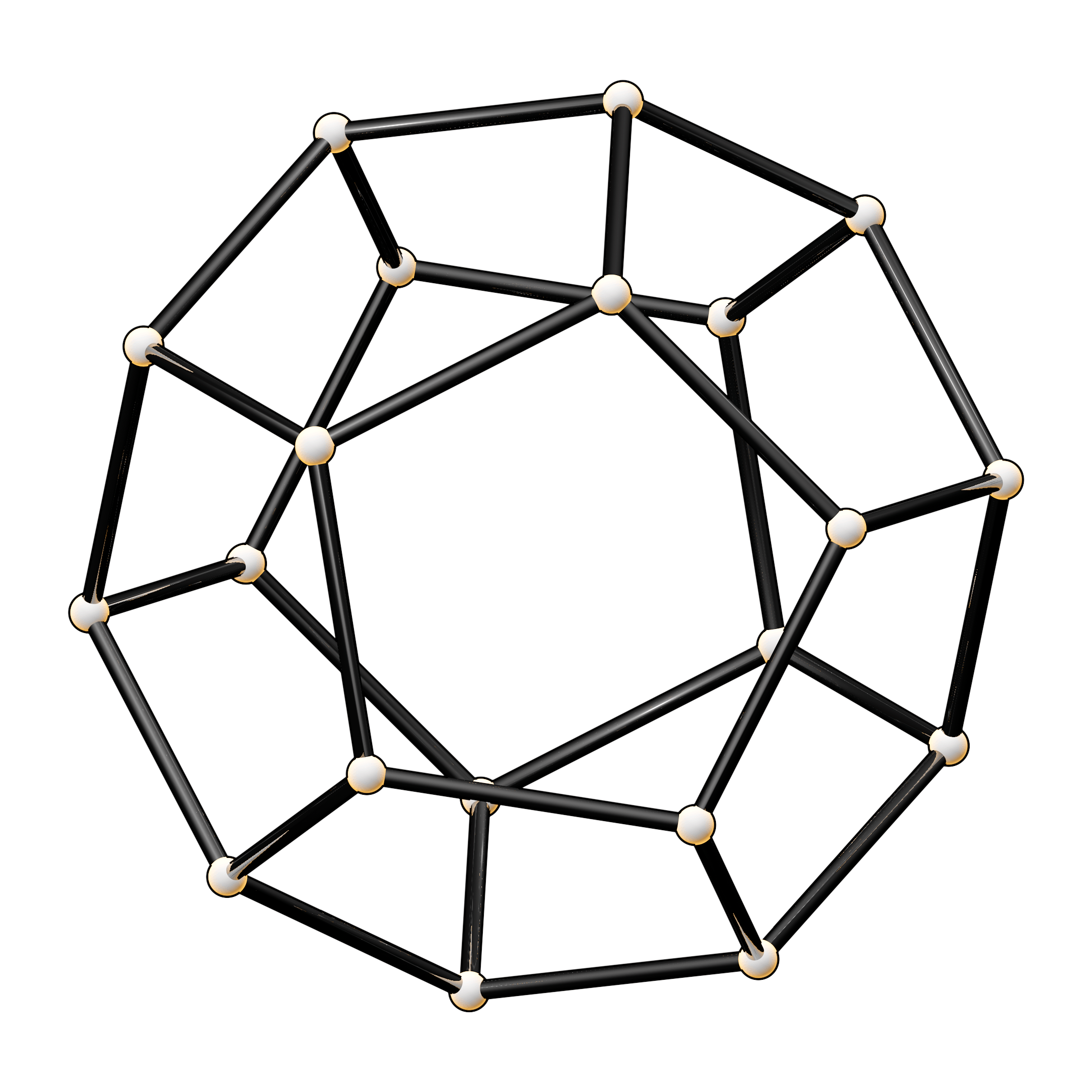}}
   \fbox{\includegraphics[height=4cm]{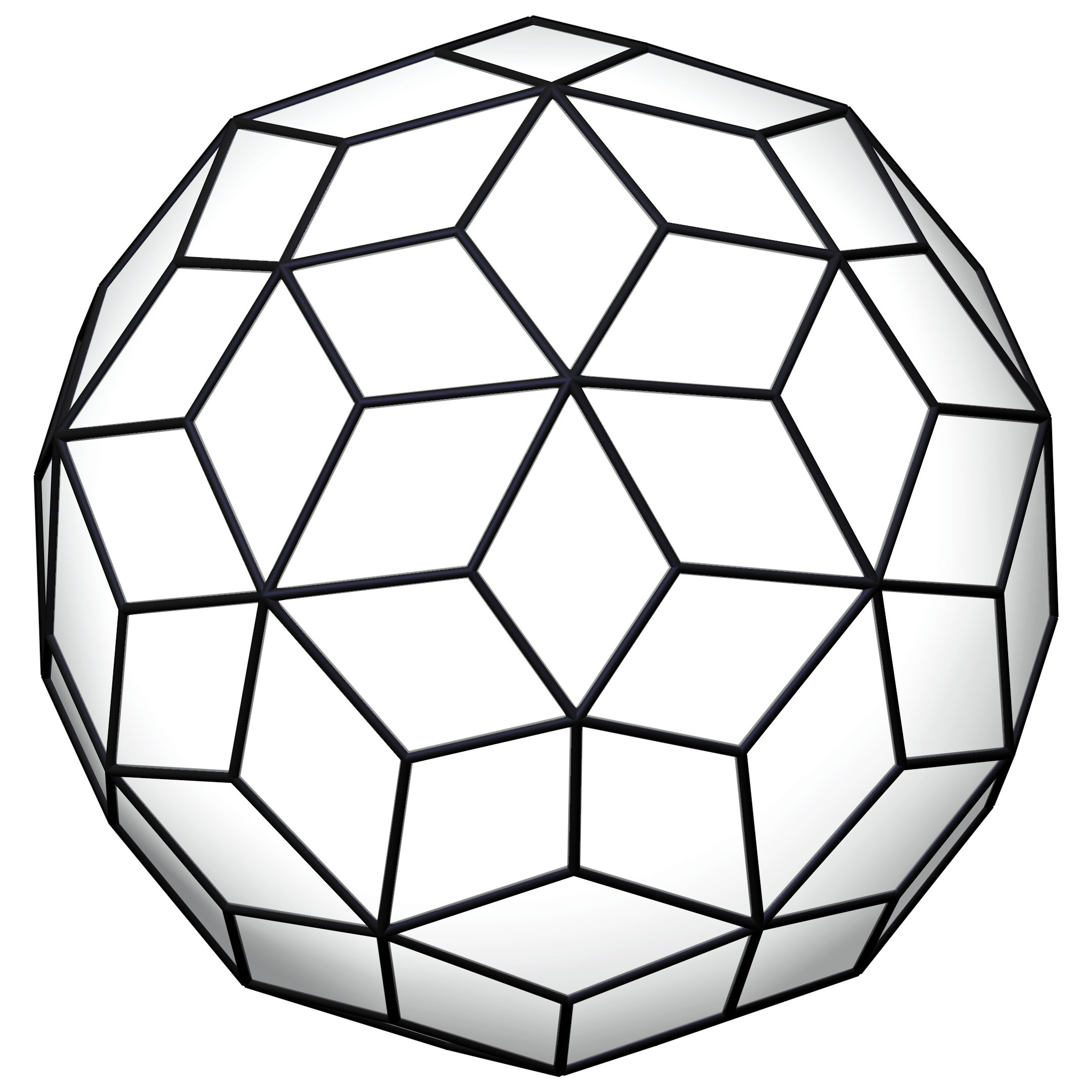}}
   \fbox{\includegraphics[height=4cm]{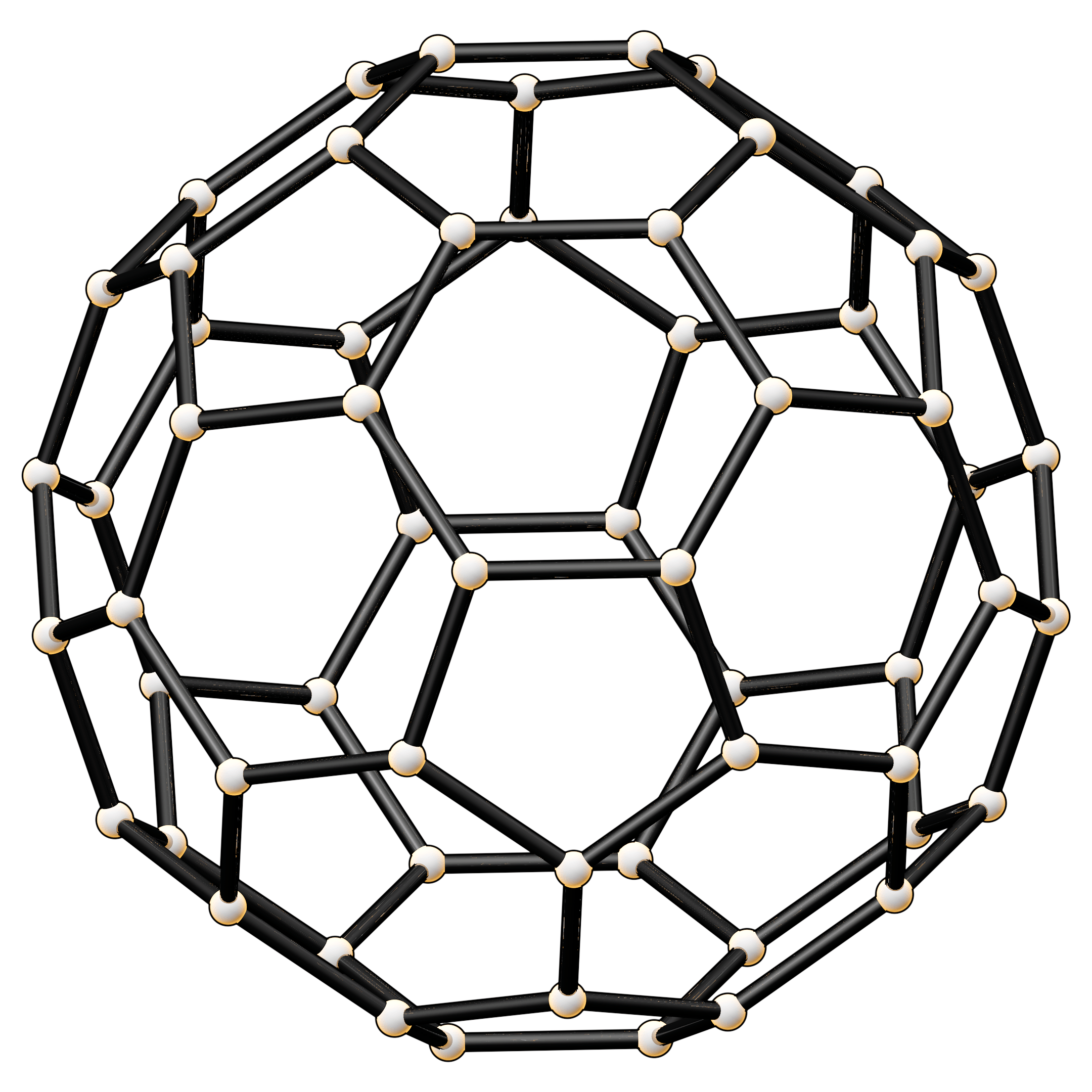}}
    \caption{\small [left] The rhombic triacontahedron RT30 with its 
    dodecahedral constraint graph contrasts with [right] the rhombic 
    enneacontahedron RE90 with its truncated-icosahedral constraint 
    graph; both have pentagonal rosettes but RE90 also has
    hexagonal cycles without holonomy.}
    \label{fig:zonohedra}
\end{figure}

The \style{rhombic enneacontahedron} (RE90), with 90 rhombic faces, has richer structure (Figure~\ref{fig:zonohedra}[right]). Its 60 degree-3 vertices and 90 constraint edges form the truncated icosahedral graph -- the ``soccer ball'' pattern with 12 pentagonal and 20 hexagonal faces. The pentagons yield frustrated 5-cycles; the hexagons admit consistent alternation. Local impossibility at the 5-fold axes coexists with local consistency at the 3-fold axes, yet no global assignment of convex and concave corners exists. The $H^1$ obstruction, concentrated at the pentagons, cannot be resolved by any choice on the hexagons.

\subsection{Twisted rings and the axis twist class}
\label{sec:twisted}

The preceding examples derive their holonomy from the coupling 
cochain $\coupling$ alone: opposition constraints along an odd cycle 
sum to $1 \in \Z_2$. But the coupling has been doing more work than 
it appears. To say that $\coupling(e) = 1$ -- that adjacent states 
must oppose -- presupposes a way to \emph{compare} the two state 
spaces: ``clockwise at $u$'' and ``clockwise at $v$'' must mean the 
same thing. For planar gears, a fixed normal to the plane provides 
this comparison canonically. But for spinning objects whose rotation 
axes vary in space, ``clockwise'' depends on a choice of axis 
orientation via the right-hand rule, and when these axes rotate 
gradually around a loop, the naming of states can reverse upon 
return. The coupling $\coupling(e)$ encodes the \emph{total} 
comparison -- both the mechanical constraint (agree or oppose) and 
the identification of labels (which depends on the relative geometry 
of the two axes). At the cochain level, these contributions are 
entangled: decomposing $\coupling$ into a ``constraint part'' and a 
``frame part'' requires a gauge choice (a consistent assignment of 
axis orientations) that may not exist globally.

At the level of cohomology classes, however, a clean and canonical 
decomposition emerges. Define the \style{twist class} 
$\omega(\gamma) \in H^1(\gamma; \Z_2)$ of the axis field around a 
cycle $\gamma$ as follows: $\omega(\gamma) = 0$ if one can 
consistently orient all the rotation axes around $\gamma$ (choosing 
``up'' at each vertex so that the orientation varies continuously), 
and $\omega(\gamma) = 1$ if one 
cannot.\footnote{The twist class has a 
name: it is the first \emph{Stiefel-Whitney class} $w_1(L)$ of the line bundle 
$L \to \gamma$ spanned by the unoriented rotation axes. When axes 
are encoded by a map $f: \gamma \to \mathbb{RP}^2$, one has 
$\omega = f^*(w_1)$ where $w_1$ generates 
$H^1(\mathbb{RP}^2; \Z_2)$. The reader unfamiliar with 
characteristic classes loses nothing by treating $\omega$ as what it 
is: a parity check on axis orientability.}
This is the same kind of parity test the reader has already seen in 
the holonomy criterion -- applied now to axis orientations rather 
than spin states.

Let $[\coupling_0]$ denote the \style{flat-frame class}: the 
cohomology class the coupling would represent if all axes were 
parallel. For pure agreement, $[\coupling_0] = 0$; for pure 
opposition on an $n$-cycle, $[\coupling_0] = n \bmod 2$. The 
total holonomy decomposes as
\[
[\coupling] \;=\; [\coupling_0] \;+\; \omega \quad \in H^1(\gamma; \Z_2).
\]
Both terms are well-defined cohomology classes -- gauge-invariant -- 
even though no canonical cochain-level decomposition of $\coupling$ 
exists. This is the virtue of working with classes rather than 
representatives: the ambiguity in distributing holonomy between 
constraint and frame disappears upon passage to $H^1$.

\begin{figure}[htb]
    \centering
    \setlength{\fboxsep}{0pt}
    \fbox{\includegraphics[width=3.15in]{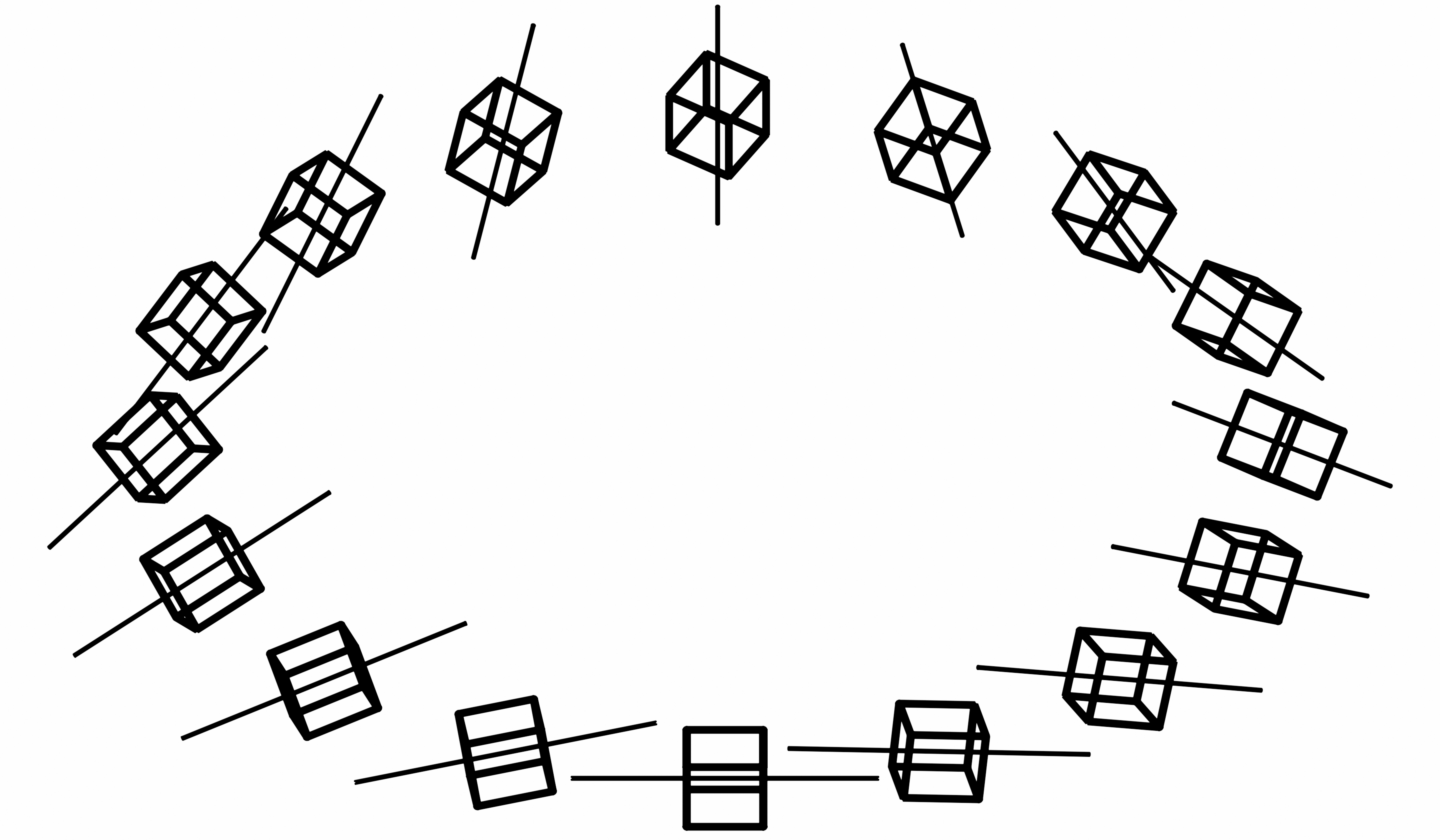}}
    \fbox{\includegraphics[width=3.15in]{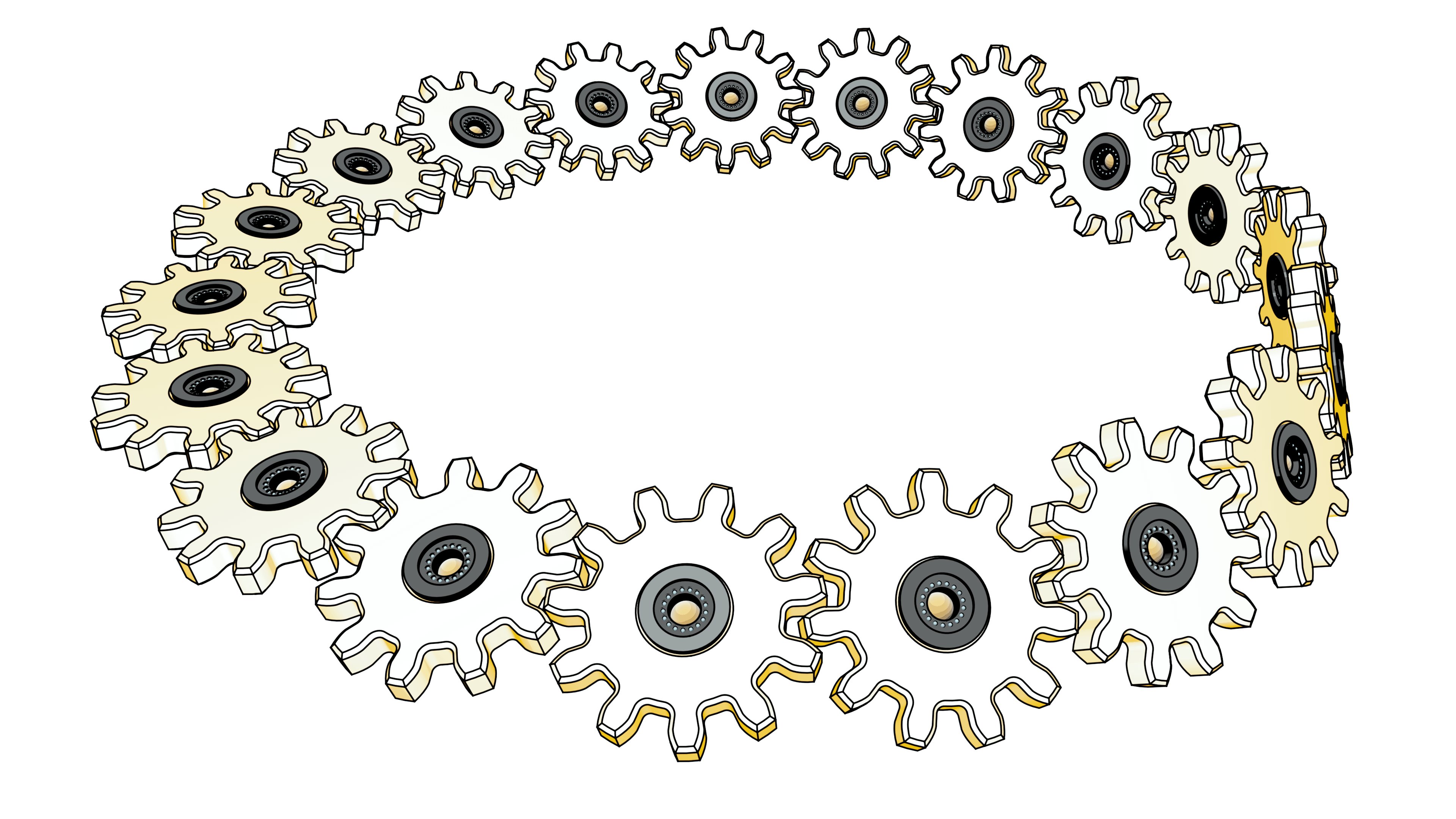}}
    \caption{\small Two twisted systems: [left] the spinning Necker ring is 
    a cycle of Necker cubes, each spinning about a local axis that rotates 
    gradually around the ring, completing a half-turn over one circuit; [right] the
    gear equivalent is the well-known M\"obius gear system \cite{Hoover2006}, 
    in which an odd number of gears with a global twist has vanishing holonomy.}
    \label{fig:necker-ring}
\end{figure}

Consider both coupling types on an $n$-cycle whose axis field 
completes a half-turn ($\omega = 1$):

\emph{Agreement coupling} ($[\coupling_0] = 0$): Necker cubes 
spinning about axes that rotate gradually around the ring 
(Figure~\ref{fig:necker-ring}[left]). With parallel axes the system 
is consistent; the half-twist contributes $\omega = 1$, giving 
$[\coupling] = 0 + 1 = 1$. The geometric twist \emph{creates} 
impossibility.

\emph{Opposition coupling} ($[\coupling_0] = n \bmod 2$): gears 
on a M\"obius strip (Figure~\ref{fig:necker-ring}[right]). On a 
planar ring, $[\coupling] = n \bmod 2$: odd rings lock. With the 
half-twist, $[\coupling] = n + 1 \bmod 2$: odd rings now spin 
freely, and even rings lock. The geometric twist \emph{resolves} 
impossibility.

The M\"obius gear ring is a known mechanical curiosity -- an odd 
number of meshing gears, free to rotate despite the usual parity 
obstruction \cite{Hoover2006}. The explanation is now immediate: the 
twist class $\omega = 1$ contributes a unit of holonomy that cancels 
the combinatorial parity of the opposition constraints. The same 
$\omega$, via the same addition in $H^1$, creates impossibility for 
agreement coupling and resolves it for opposition coupling.

\begin{remark}
The Necker depth bistability -- forward-facing or backward-facing 
interpretation -- is a second $\Z_2$ variable on each spinning cube, 
independent of spin direction. The depth labels form a trivial local 
system: one can maintain a consistent depth reading throughout the 
ring regardless of axis geometry. The spin paradox lives entirely in 
the rotation, not the depth: two $\Z_2$ structures coexist on the 
same geometric substrate, one with $\omega = 0$ and one with 
$\omega = 1$.
\end{remark}

\subsection{Torsors and double covers}
\label{sec:torsors}

The algebraic structure underlying these examples admits a geometric interpretation essential for the visualization techniques of \S\ref{sec:MoMA}. A coupling $\coupling \in C^1(\cgraph; \Z_2)$ on a connected graph defines a \style{double cover}: take two copies of each vertex, and for each edge $e = uv$, connect the copies according to $\coupling(e)$. If $\coupling(e) = 0$, connect $0_u \leftrightarrow 0_v$ and $1_u \leftrightarrow 1_v$ (agreement); if $\coupling(e) = 1$, connect $0_u \leftrightarrow 1_v$ and $1_u \leftrightarrow 0_v$ (opposition).

The resulting space is the total space of a $\Z_2$-torsor over $\cgraph$. When $[\coupling] = 0$, this space is disconnected -- two disjoint copies of $\cgraph$ -- and choosing a component amounts to choosing a global state. When $[\coupling] \neq 0$, the total space is connected: a twisted cover where traversing any odd-holonomy cycle returns you to the opposite sheet. No consistent choice of sheet is possible; no global section exists.
This construction realizes a fundamental classification: 
\begin{quote}
    $\Z_2$-torsors over a connected complex $X$ correspond bijectively to classes in $H^1(X; \Z_2)$.
\end{quote}
The trivial class corresponds to the disconnected double cover $X \times \Z_2$ 
and nontrivial classes to connected covers. The precise statement appears as 
Theorem~\ref{thm:torsor-classification} in Appendix~\ref{sec:appendix-cohomology}.

\begin{remark}[Torsors on graphs versus complexes]
\label{rem:torsor-convention}
The double-cover construction defines a torsor over any graph $\cgraph$, requiring only the 1-skeleton. When $\cgraph = X^{(1)}$ is the full 1-skeleton of a 2-complex and $\coboundary\coupling = 0$ (flatness), the torsor extends over $X$. When $\coboundary\coupling \neq 0$, curvature obstructs this extension: we have a torsor over $X^{(1)}$ but no flat extension over $X$. This distinction becomes essential in \S\ref{sec:relH2}, where curvature defects are precisely the obstruction to extending a 1-skeleton torsor over 2-cells.
\end{remark}

For the circle $S^1$, the classification yields exactly two torsors: the trivial cover (two disjoint circles) and the connected double cover -- the boundary of a M\"obius band, wrapping twice around the base. This suggest that for an example such as an odd cycle of planar gears (defining a nontrivial torsor on the circle) there should be a way to ``see'' the torsor via a double cover. This perspective -- torsors as twisted coverings, monodromy measuring the failure to close after transport around a cycle -- provides the conceptual foundation for visualizing paradox through animation.

\section{The Method of Monodromic Apertures (MoMA)}
\label{sec:MoMA}

If one covers part of an odd gear ring with an opaque mask, leaving a window through which an interval of adjacent gears are visible, then within this window, the gears spin freely -- the constraint equation is solvable on any contractible subset. Now slide the window continuously around the ring, updating the displayed spins to maintain local consistency at each position (see Figure \ref{fig:single-aperture}[left]). After one complete circuit, the window returns to its starting position, but the displayed configuration has flipped: gears that initially appeared clockwise now appear counterclockwise. This flip is the monodromy of the nontrivial $\Z_2$-torsor over $S^1$, made visible through animation.

Static analysis of $\coboundary x = \coupling$ determines whether global sections exist: compute holonomy around each cycle, and if any is nonzero, the system is paradoxical. But this analysis has a limitation. The classification theorem (Theorem~\ref{thm:torsor-classification}) guarantees $|H^1(X; \Z_2)|$ distinct torsors over any space $X$ -- one per cohomology class. The coupling $\coupling$ selects one via $[\coupling]$; the others exist as mathematical objects, invisible to holonomy computation, unchosen by the physics.

The Method of Monodromic Apertures exhibits torsor structure directly, converting holonomy (an algebraic computation) into monodromy (a visible flip) after transport around a loop. A single aperture reveals the torsor selected by $\coupling$. To exhibit the ``hidden'' torsors corresponding to other cohomology classes requires multiple apertures subject to observational constraints, constructing new torsors over configuration space.

\subsection{Apertures and local sections}
\label{sec:apertures}

An \style{aperture} is a contractible viewing window $U \subset X$ through which we observe the bistable system. Since any $\Z_2$-torsor trivializes over a contractible set, the restriction $\torsor|_U \cong U \times \Z_2$ admits exactly two local sections -- two ways to assign a consistent state throughout the window. The aperture \style{displays} one of these sections: a choice of sheet in the trivial double cover over $U$.
As the aperture moves through $X$, the displayed section updates by parallel transport.

\begin{lemma}[Aperture monodromy]
\label{lem:aperture-monodromy}
Let $\gamma: [0,1] \to X$ be a closed loop along which an aperture moves, with initial display $s_0 \in \pi^{-1}(\gamma(0))$. Transport of $s_0$ along $\gamma$ returns to the same sheet if and only if the torsor has trivial monodromy along $\gamma$; otherwise the displayed state returns globally flipped.
\end{lemma}

\noindent
The proof is immediate from covering space theory: a loop lifts to a loop if and only if it lies in the kernel of the monodromy homomorphism $\pi_1(X) \to \Z_2$.

\subsection{Single apertures on odd cycles}
\label{sec:single-odd}

The opening example -- a sliding window on an odd gear ring -- instantiates this lemma. The torsor is nontrivial ($[\coupling] \neq 0$), so transport around the generating cycle returns the opposite sheet. The animation makes the covering space structure perceptible: what was algebraically certified as ``nonzero holonomy'' becomes visually manifest as ``the spins flipped.''

The same technique applies to odd rosettes in rhombic tilings (Figure~\ref{fig:single-aperture}). Slide an aperture around a P3 pentagonal rosette, displaying convex/concave assignments that satisfy local constraints. After one circuit, the displayed assignment has reversed. The static picture shows inconsistency; the animation exhibits the underlying torsor.

\begin{figure}[htb]
    \centering
    \setlength{\fboxsep}{0pt}
    \fbox{\includegraphics[height=6cm]{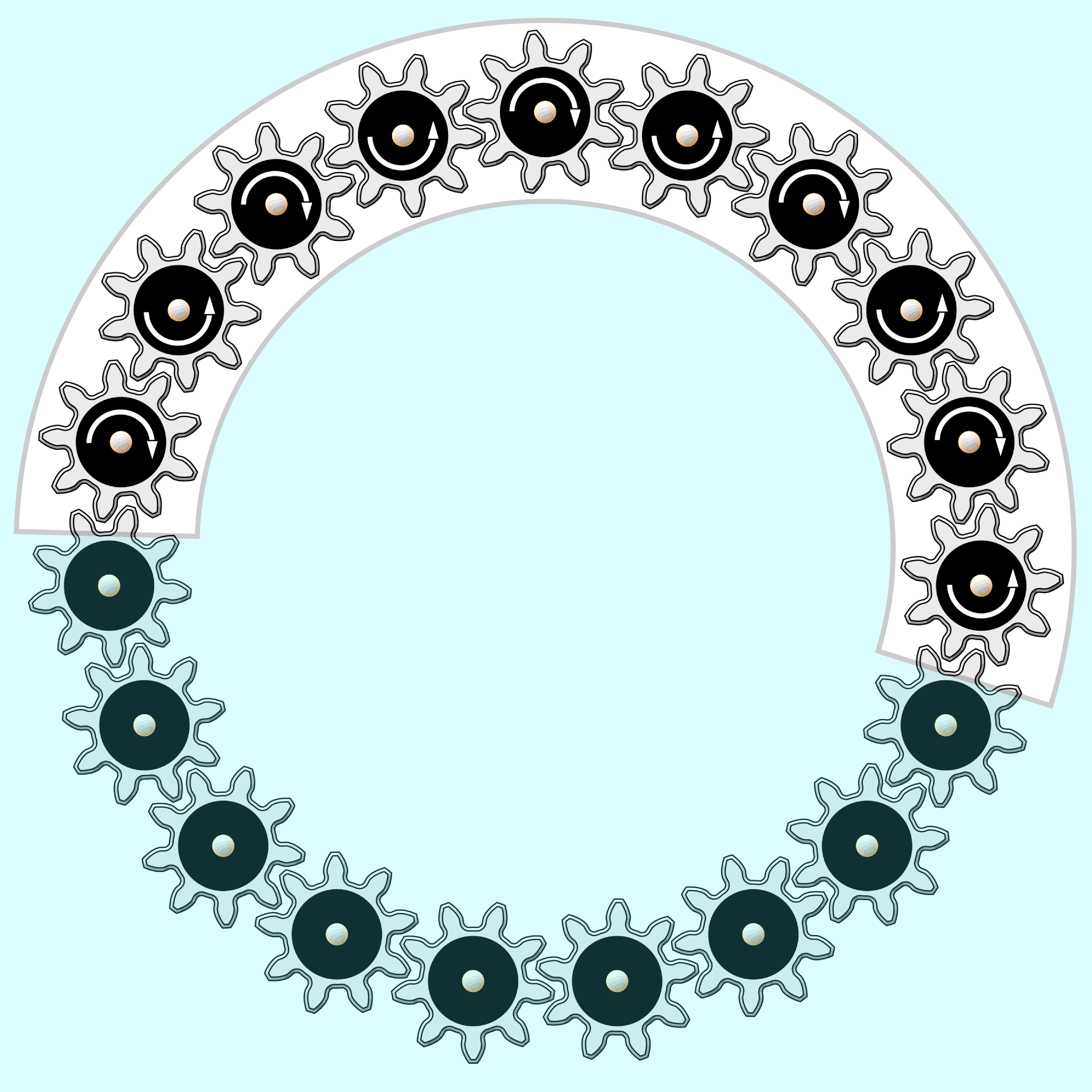}}
    \hspace{0.5em}
    \fbox{\includegraphics[height=6cm]{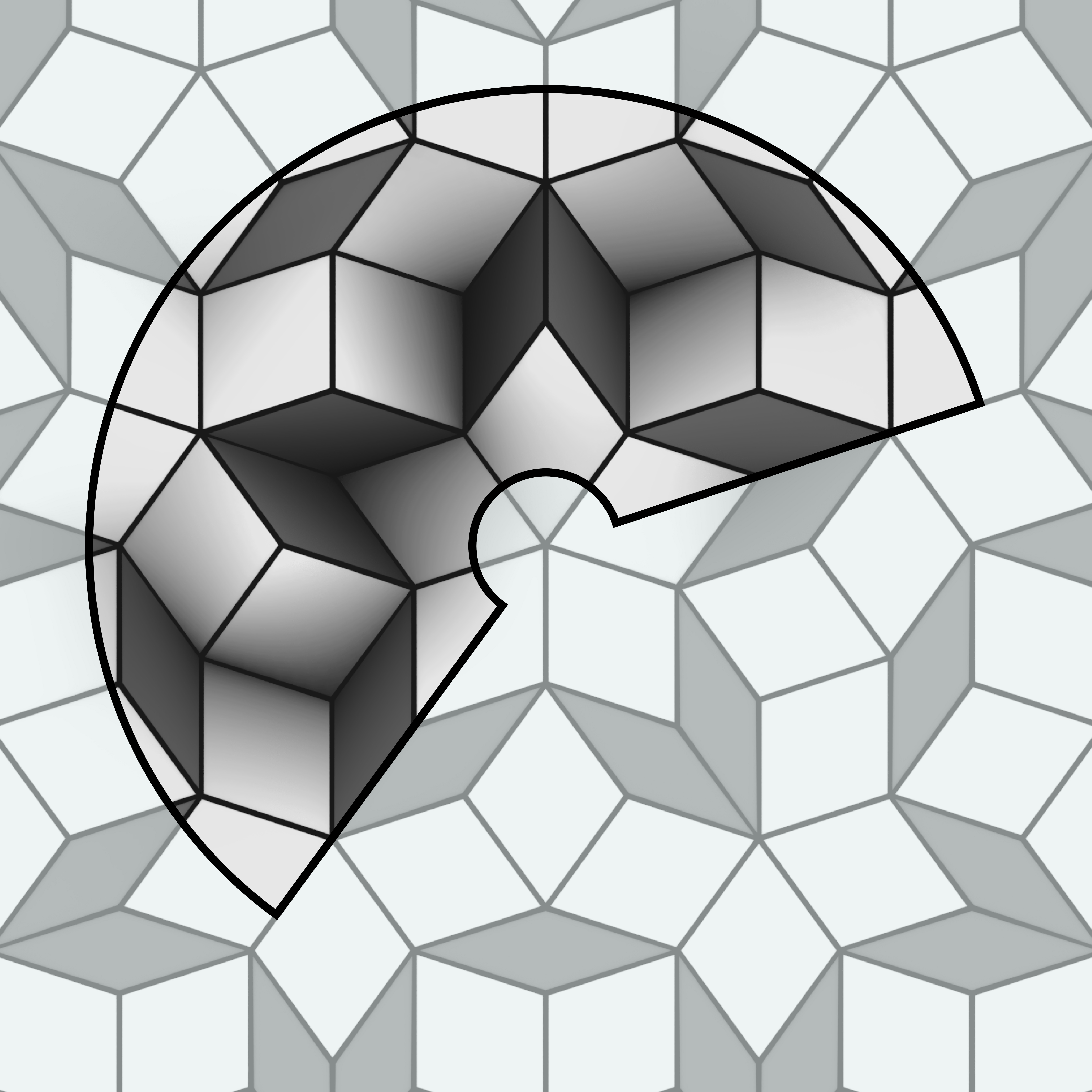}}
    \caption{\small Single-aperture MoMA on odd cycles. [Left] A window on an odd gear ring permits local spinning; sliding this aperture around the loop yields nontrivial monodromy. [Right] The same technique on a P3 quasitiling -- shading within in a window and rotating it one turn -- reveals the torsor structure of convex/concave assignments.}
    \label{fig:single-aperture}
\end{figure}

\subsection{Dual apertures and configuration space}
\label{sec:dual-apertures}

A single aperture on a cycle with trivial holonomy reveals nothing: the displayed section returns unchanged, reflecting the trivial monodromy of the torsor selected by $\coupling$. Yet $H^1(S^1; \Z_2) \cong \Z_2$ guarantees two distinct torsors over any loop. The coupling with zero holonomy selects the trivial one, but the nontrivial torsor exists as a mathematical object, unchosen by the physics.

To exhibit this hidden torsor, we employ two disjoint apertures subject to an \style{opposite-spin rule}: the displayed local sections must differ. If one aperture shows clockwise (or forward, or convex), the other shows counterclockwise (or backward, or concave). This rule constructs a new $\Z_2$-torsor -- not over $X$ itself, but over the \style{configuration space} of aperture pairs.

The relevant space is the unordered configuration space
\[
\conf_2(X) = \bigl\{ \{p, q\} \subset X : p \neq q \bigr\},
\]
the space of 2-element subsets. For $X = S^1$, this configuration space is a M\"obius band: the ordered configuration space $(S^1 \times S^1)\setminus \Delta$ is an open cylinder, and quotienting by the swap $(p,q) \mapsto (q,p)$ yields a M\"obius band with $\pi_1(\conf_2(S^1)) \cong \Z$.

At each configuration $\{p, q\}$, the opposite-spin rule demands complementary displays. Concretely, this amounts to choosing which aperture shows state 0 and which shows state 1 -- but since apertures are unlabeled, only the partition matters. There are exactly two such partitions, exchanged by the global $\Z_2$-action. The space of these choices, fibered over $\conf_2(X)$, defines a $\Z_2$-torsor $\tilde{\torsor} \to \conf_2(X)$.

\subsection{The exchange loop}
\label{sec:exchange}

The monodromy of the dual-aperture torsor is detected by the \style{exchange loop}: a path in $\conf_2(X)$ along which the two aperture positions smoothly interchange.

Let $\{p, q\}$ be an initial configuration with $p$ and $q$ connected by a path $\gamma$ in $X$. Slide $p$ along $\gamma$ toward $q$ while simultaneously sliding $q$ toward $p$, the two passing through each other's initial positions. Since the configuration space is unordered, the endpoint $\{q, p\} = \{p, q\}$: we have traced a closed loop in $\conf_2(X)$.

\begin{lemma}
\label{lem:exchange-monodromy}
The exchange loop has nontrivial monodromy in the dual-aperture torsor $\tilde{\torsor} \to \conf_2(X)$.
\end{lemma}

\begin{proof}
Parallel transport preserves each aperture's displayed state; the exchange swaps which aperture occupies which position; hence it transposes the assignment of states to positions.
\end{proof}

For a cycle ($X \simeq S^1$), the configuration space $\conf_2(X)$ is an open M\"obius band, homotopy equivalent to $S^1$. The exchange loop generates $H^1(\conf_2(X)) \cong \Z_2$, and its nontrivial monodromy classifies the dual-aperture torsor as the connected double cover. This exhibits the hidden nontrivial torsor that was invisible to single-aperture observation: see the Necker cube and gear instantiations of Figure \ref{fig:dual-aperture}.

\begin{figure}[htb]
    \centering
    \setlength{\fboxsep}{0pt}
    \fbox{\includegraphics[height=6cm]{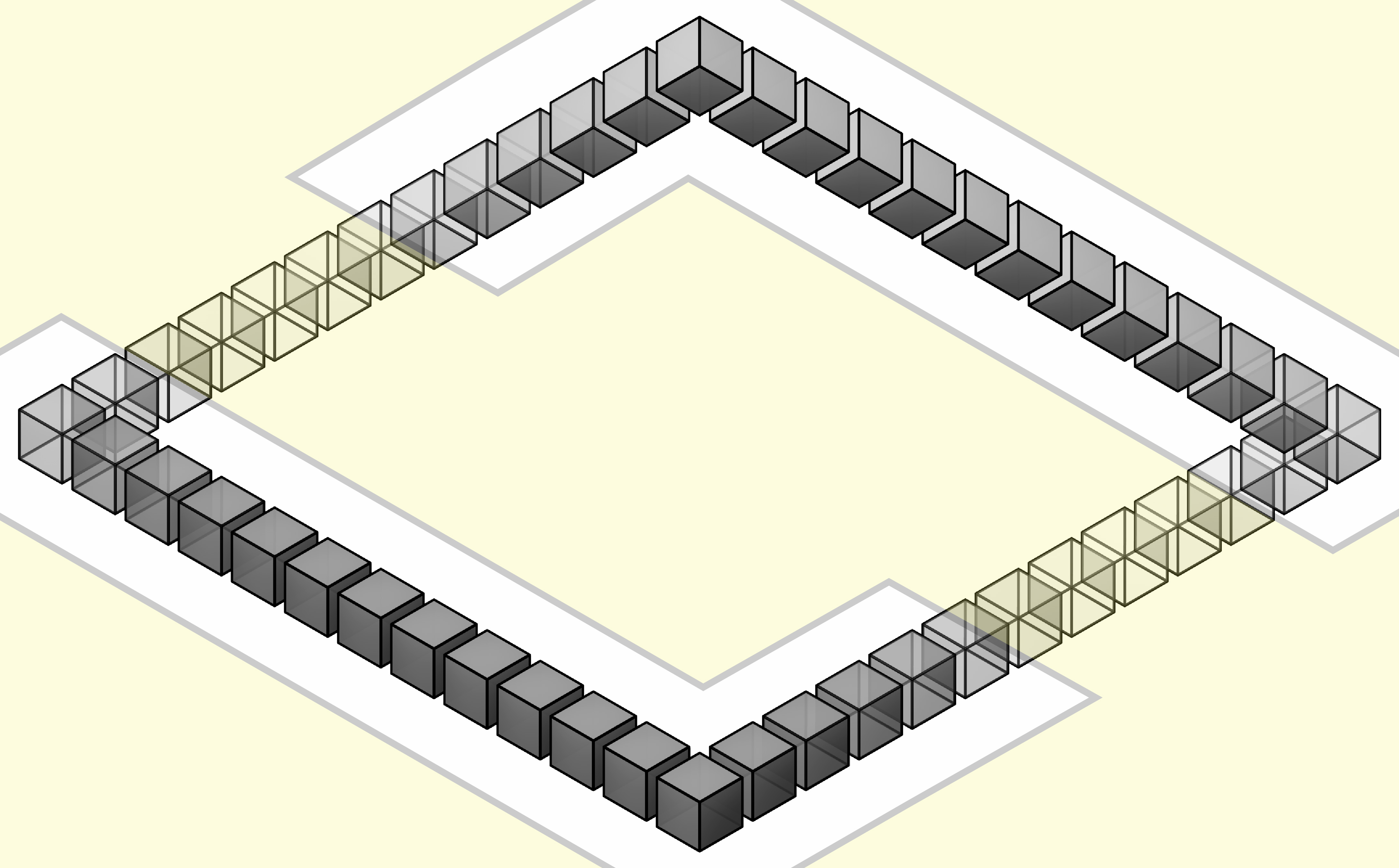}}
    \fbox{\includegraphics[height=6cm]{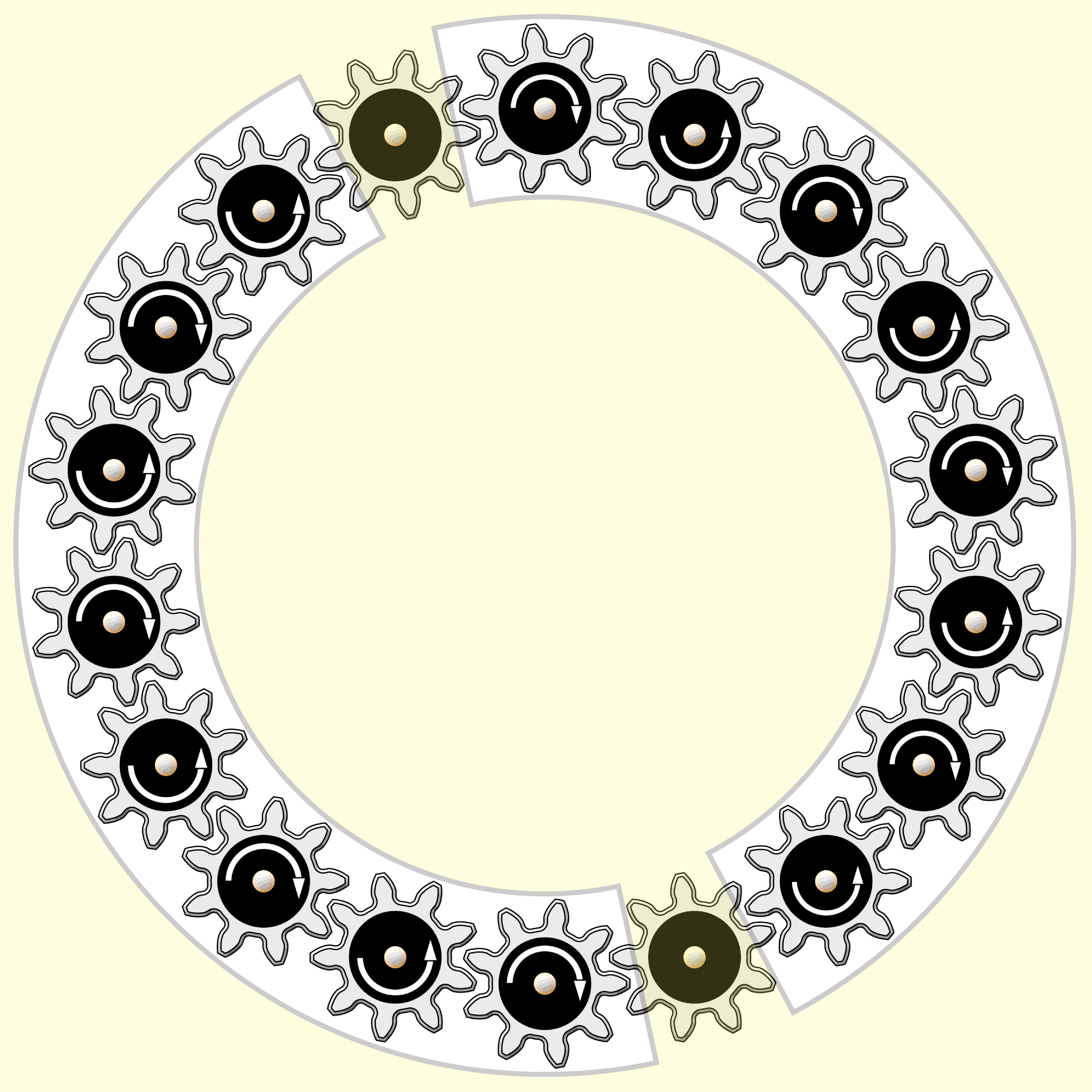}}
    \caption{\small Dual-aperture MoMA with the opposite-spin rule. [Left] A ring of Necker cubes with antipodal sliding apertures; the exchange loop reveals nontrivial monodromy on configuration space. [Right] An even gear cycle -- trivial holonomy, but dual apertures exhibit the hidden nontrivial torsor upon exchange.
}
    \label{fig:dual-aperture}
\end{figure}

The construction extends to two-dimensional fields. Figure~\ref{fig:lozenge-dual-aperture} shows a lozenge tiling variant of the Necker interval of Figure \ref{fig:necker-interval}. There are two apertures (antipodal half-spaces) displaying opposite stepped-surface orientations, the intervening strip fading to flat ambiguity. As the entire aperture-strip decomposition rotates around the figure center, the apertures trace the exchange loop in the configuration space $\conf_2(\R^2)$ of unordered pairs in the plane, which has fundamental group $\Z$ and so supports the same exchange monodromy as the circular case.

\begin{figure}[htb]
    \centering
    \setlength{\fboxsep}{0pt}
    \fbox{\includegraphics[height=6cm]{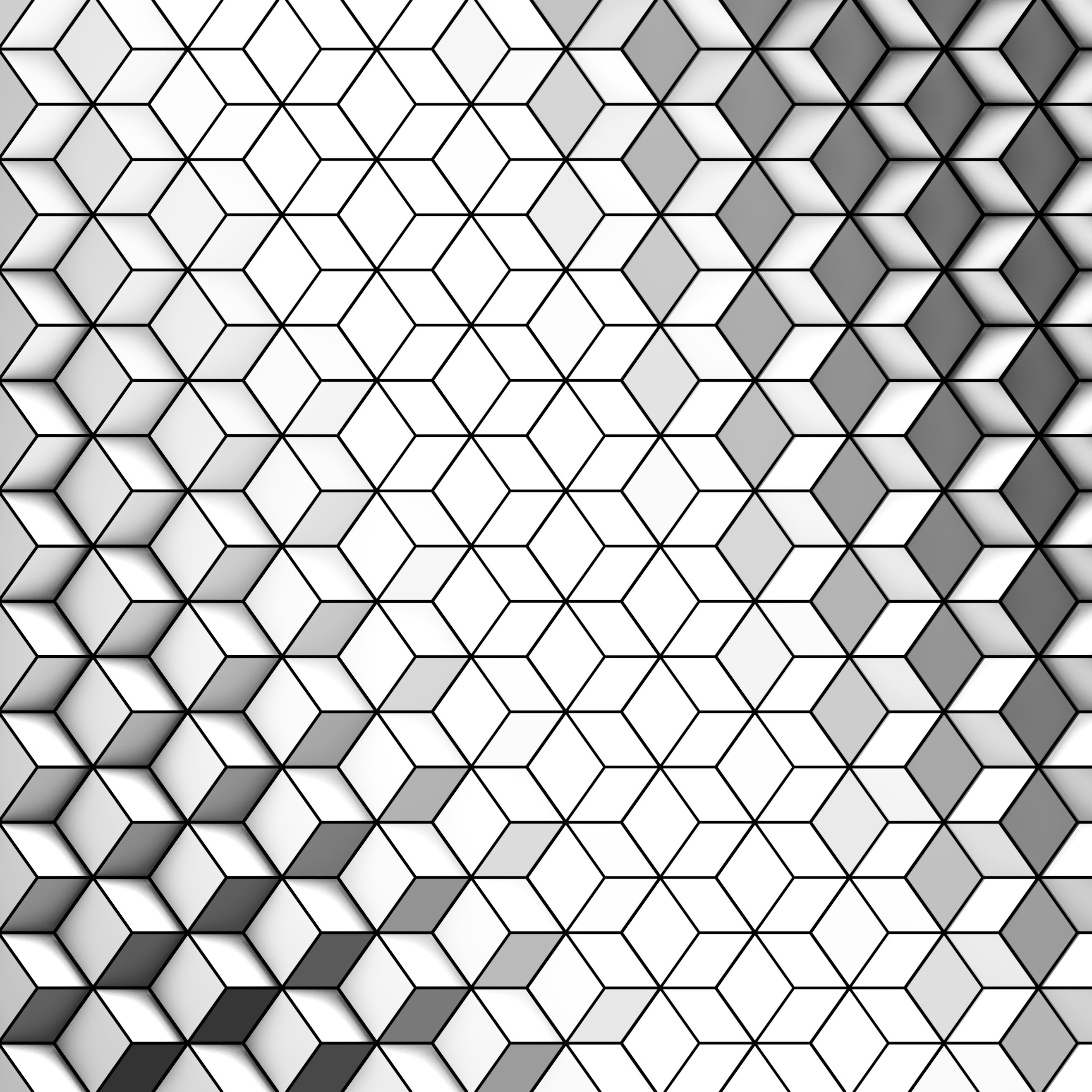}}
    \hspace{2em}
    \fbox{\includegraphics[height=6cm]{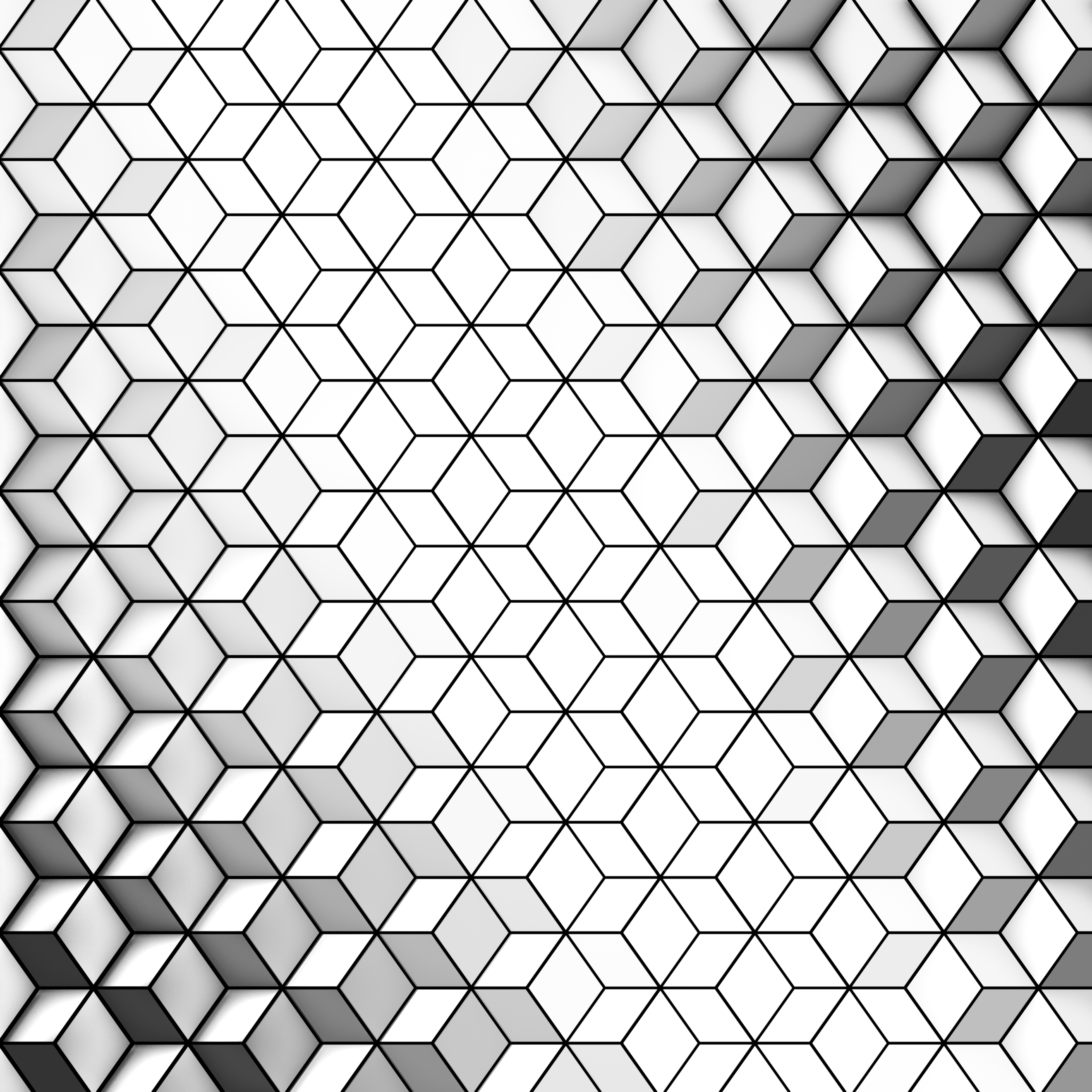}}
\caption{\small A lozenge tiling with dual half-plane apertures separated by a strip of ambiguity. Rotating the configuration a half-turn traces the exchange loop; after an exchange, the displayed orientations (as interpreted via shading to orient normals to faces) have swapped positions.
}
    \label{fig:lozenge-dual-aperture}
\end{figure}

\subsection{On terminology: holonomy vs. monodromy}
\label{sec:hol-vs-mono}

Given the centrality of {\em holonomy} in this paper and the use of {\em monodromy} in the MoMA moniker, it is worth recalling the [subtle, often obscured] difference between the two: holonomy is to monodromy as noun is to verb.

\style{Holonomy} is a group element -- the measurement obtained by summing the coupling cochain around a cycle: $\mathrm{hol}_\gamma(\coupling) = \sum_{e \in \gamma} \coupling(e) \in \Z_2$. Different cochains in the same cohomology class may have different edge values but yield identical holonomy on every cycle. Holonomy is algebraic, computational, static: it answers ``what is the total mismatch?''

\style{Monodromy} is an action -- what happens to fiber elements when transported around a loop. Given a torsor $\torsor \to X$ and a loop $\gamma$, lift $\gamma$ to the total space starting at some point $p \in \torsor$; monodromy asks where you arrive upon return. Trivial monodromy means the lift closes; nontrivial monodromy means you land on the opposite sheet. Monodromy is geometric, dynamic, experiential: it answers ``what happens when I walk around?''
The relationship is direct: holonomy \emph{determines} monodromy. If $\coupling$ presents the torsor $\torsor$, then nonzero holonomy around $\gamma$ implies nontrivial monodromy along $\gamma$, and conversely. This is the content of the classification theorem (Theorem~\ref{thm:torsor-classification}).

Single-aperture MoMA converts the algebraic noun into the experiential verb: the holonomy of $\coupling$ becomes visible as the monodromy of displayed sections. Dual-aperture MoMA constructs a \emph{different} torsor -- one living on configuration space rather than the base -- whose monodromy reveals structure invisible to the original holonomy computation. The configuration-space torsor has its own holonomy, of course; but it is not determined by $\coupling$. It is determined by the opposite-spin rule, which manufactures a new cocycle on $\conf_2(X)$.

\section{Relative $H^2$: Curvature}
\label{sec:relH2}

The odd cycles of \S\ref{sec:absH1} were classified by $H^1(\cgraph)$: nonzero holonomy around a cycle obstructs global sections. We now examine what happens when such a cycle bounds a disc. The discrete Stokes theorem forces a relationship between boundary and interior: nonzero boundary holonomy implies nonzero total interior curvature. This is a \style{relative} $H^2$ phenomenon -- the third level of the hierarchy, where impossibility on the boundary promotes to \style{curvature} in the interior via the connecting homomorphism $\connecting: H^1(\partial D) \to H^2(D, \partial D)$.

The mechanism is identical to that which promoted incompatible $H^0$ boundary data to relative $H^1$ conflict in \S\ref{sec:relH^1}, now shifted up one degree: boundary impossibility becomes interior curvature.
Let $D \subset X$ be a cellular disc whose boundary $\partial D$ lies entirely in the constraint graph $\cgraph$. The long exact sequence of the pair $(D, \partial D)$ includes
\[
H^1(D) \xrightarrow{\mathrm{res}} H^1(\partial D) \xrightarrow{\delta^*} H^2(D, \partial D) \to H^2(D).
\]
Since $D$ is contractible, $H^1(D) = H^2(D) = 0$, and exactness forces an isomorphism
\[
\delta^*: H^1(\partial D; \Z_2) \xrightarrow{\;\cong\;} H^2(D, \partial D; \Z_2).
\]
Both groups are $\Z_2$. Every nontrivial boundary class promotes to a nontrivial interior class. This is the Stokes Principle of \S\ref{sec:hierarchy} at $k = 1$: nonzero holonomy $\sum_{e \in \partial D} \coupling(e) = 1$ forces nonzero total curvature $\sum_{f \subset D} \mu(f) = 1$. At least one face inside $D$ must be frustrated; the boundary obstruction localizes but cannot vanish.

\subsection{Examples of curvature defects}
\label{sec:curvature-defects}

The pentagonal rosette of the P3 tiling provides the paradigmatic example. Five linked corners form a cycle $\gamma$ in $\cgraph$ with nonzero holonomy representing the nontrivial class in $H^1(\gamma)$, this cycle bounding a disc $D$ in the plane. The connecting homomorphism promotes $[\coupling|_\gamma]$ to a class in $H^2(D, \partial D)$. 

The same analysis applies to rhombic zonohedra. Each pentagonal face of the RT30's dodecahedral constraint graph, or the RE90's truncated-icosahedral graph, is a 5-cycle bounding a disc on the sphere. The relative class $[\mu] \in H^2(D, \partial D)$ forces curvature inside each disc. Globally, frustrated regions pair across the sphere -- total curvature vanishes on a closed surface -- but locally, each rosette carries its relative $H^2$ defect.

Other domains exhibit the phenomenon as well. The \style{order-7 rhombille tiling} of the hyperbolic plane consists of rhombi arranged so that the constraint graph forms the regular heptagonal tiling $\{7,3\}$ (see Figure~\ref{fig:curvature-defects}). In the tiling, the arrangement of 7 ``boxes'' about the center is a version of Penrose's heptagonal stair in Figure \ref{fig:penrose-necker}[left], but with no hole inserted. In the constraint graph, each heptagon is a 7-cycle bounding a disc where Stokes forces local curvature. 

\begin{figure}[htb]
    \centering
    \setlength{\fboxsep}{0pt}
    \fbox{\includegraphics[height=6cm]{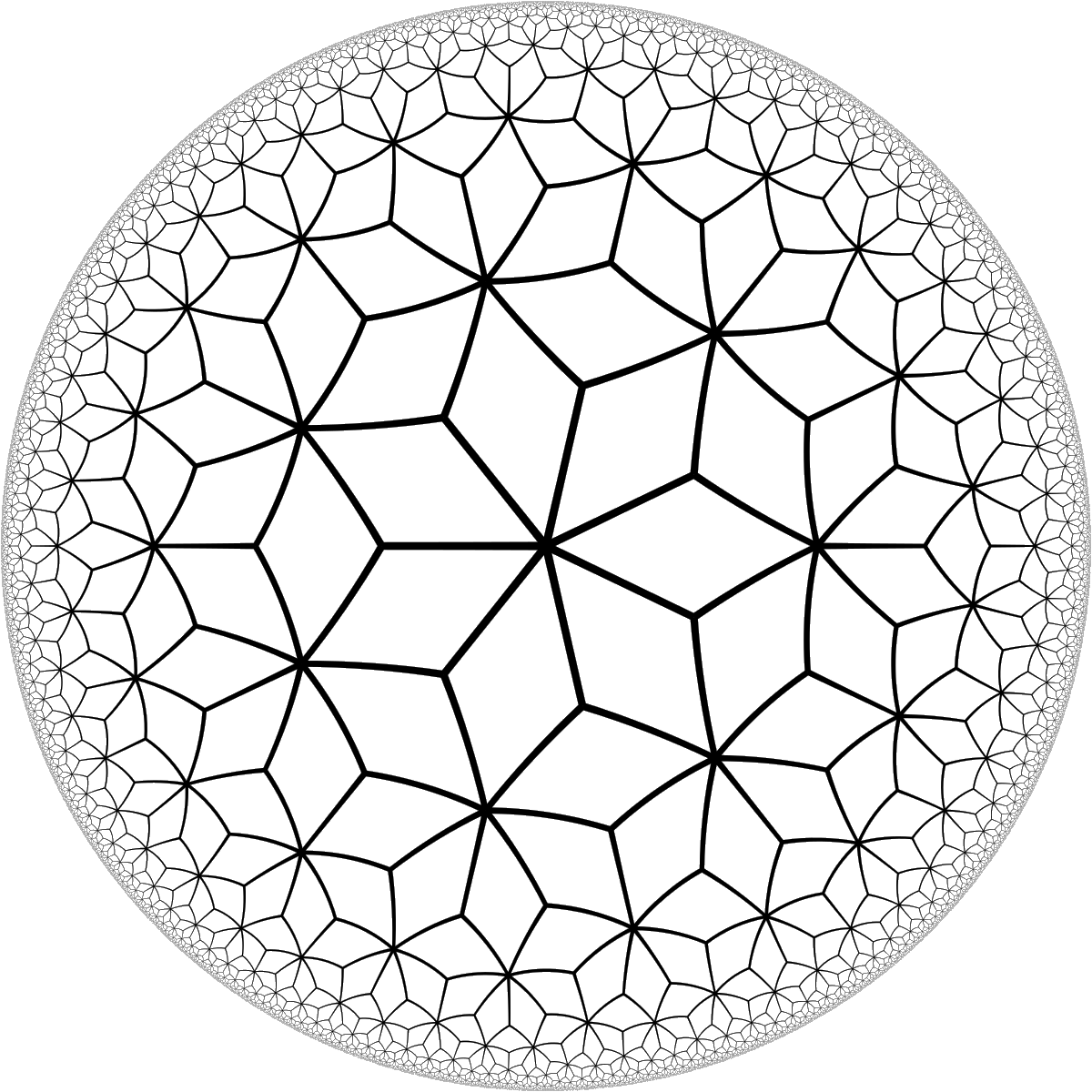}}
    \hspace{1em}
    \fbox{\includegraphics[height=6cm]{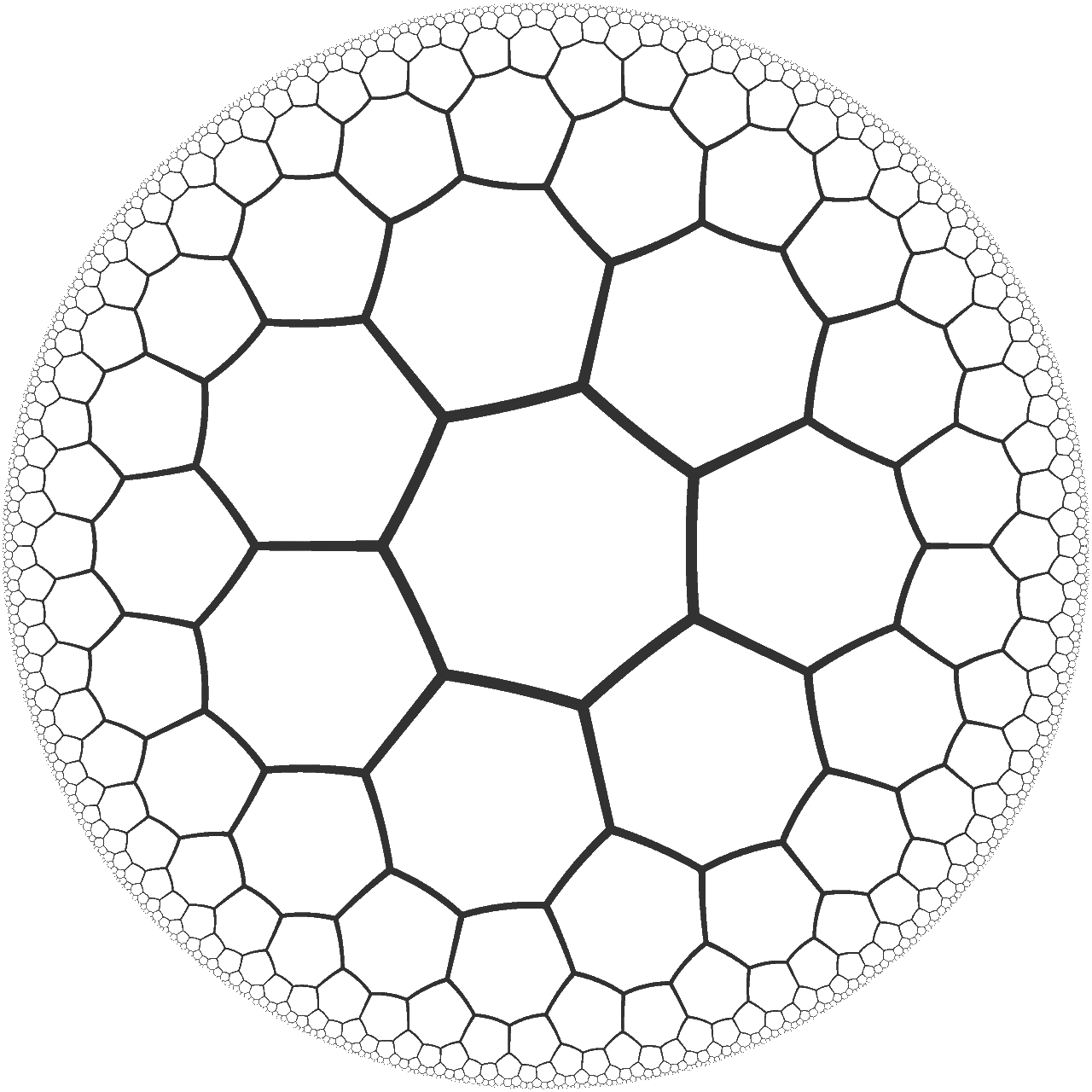}}
    \caption{\small [left] The order 7-3 rhombille tiling of the hyperbolic 
    plane. Its constraint graph is the heptagonal tiling [right] with each 7-cycle 
    bounding a disc carrying a relative $H^2$ defect.}
    \label{fig:curvature-defects}
\end{figure}

A resonant example of localized $H^2$ arises in gear configurations where curvature is more explicitly localized. Consider three quarter-planes of square gear meshes beveled at edges, meeting at a corner like three faces of an octant at the origin (Figure~\ref{fig:gear-corner}). Each quarter-plane is a bipartite grid with $\coupling \equiv 1$; each quarter-plane spins freely; any pair of quarter-planes meshes without incident. The corner is the locus of obstruction and curvature that locks the entire mesh.

\begin{figure}[htb]
    \centering
    \setlength{\fboxsep}{0pt}
    \fbox{\includegraphics[width=5in]{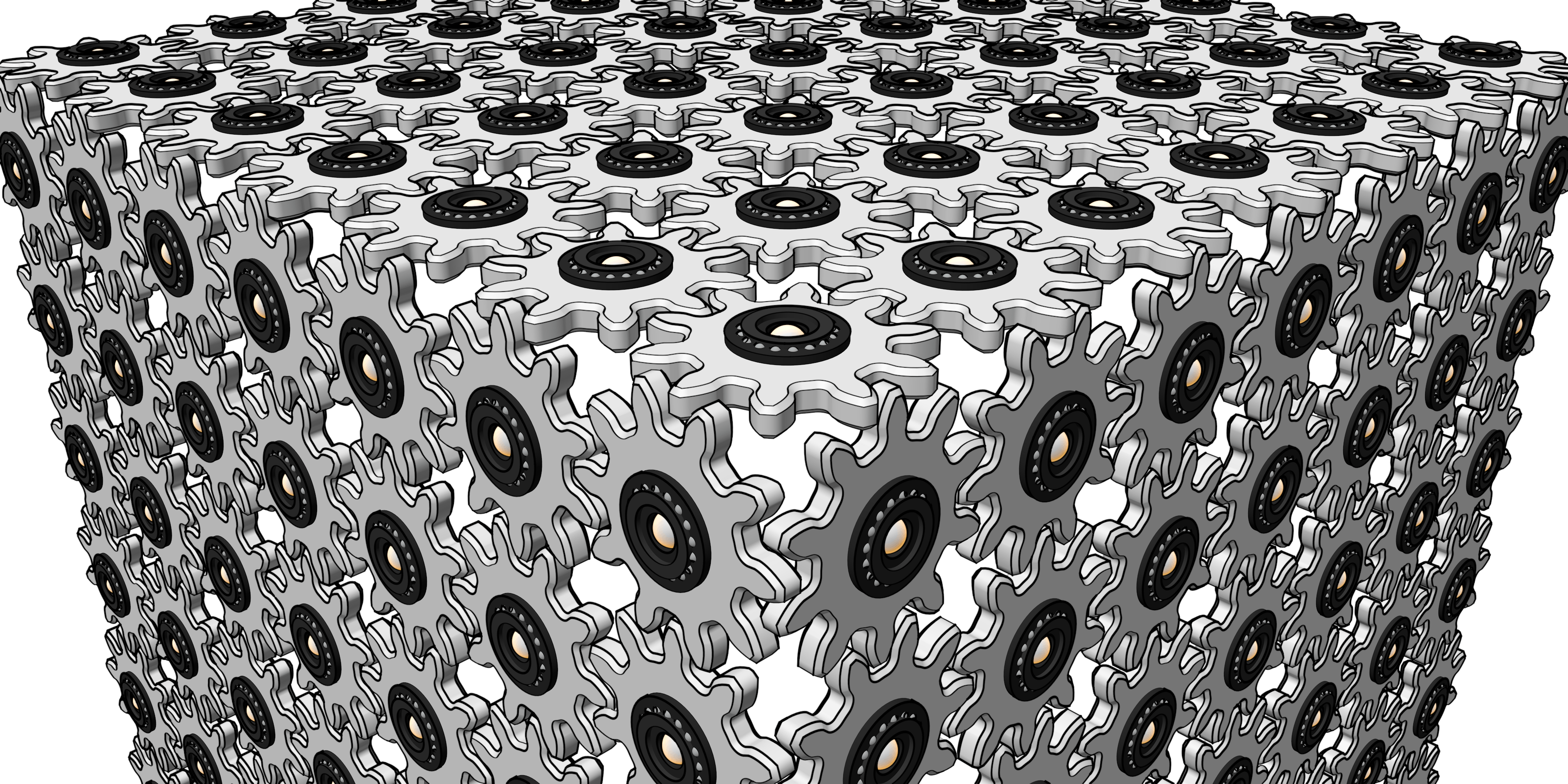}}
    \caption{\small Three quarter-planes of square gear meshes meeting at a corner. The 
    triangular corner where they meet has nontrivial boundary holonomy, 
    forcing a curvature defect.}
    \label{fig:gear-corner}
\end{figure}

The underlying base complex is illustrative: see Figure \ref{fig:gear-corner-frame}. The complex is full of quads that are (in all senses) flat, including the bevel edges between quarter-planes. The single snub simplex lies at the corner and has an odd cycle as its boundary. This place where the three gears meet is where the relative $H^2$ class is supported. Note that deleting the three corner gears does not eliminate the lock (or the cohomology): a larger odd cycle carries the obstruction.

\begin{figure}[htb]
    \centering
    \setlength{\fboxsep}{0pt}
    \fbox{\includegraphics[width=5in]{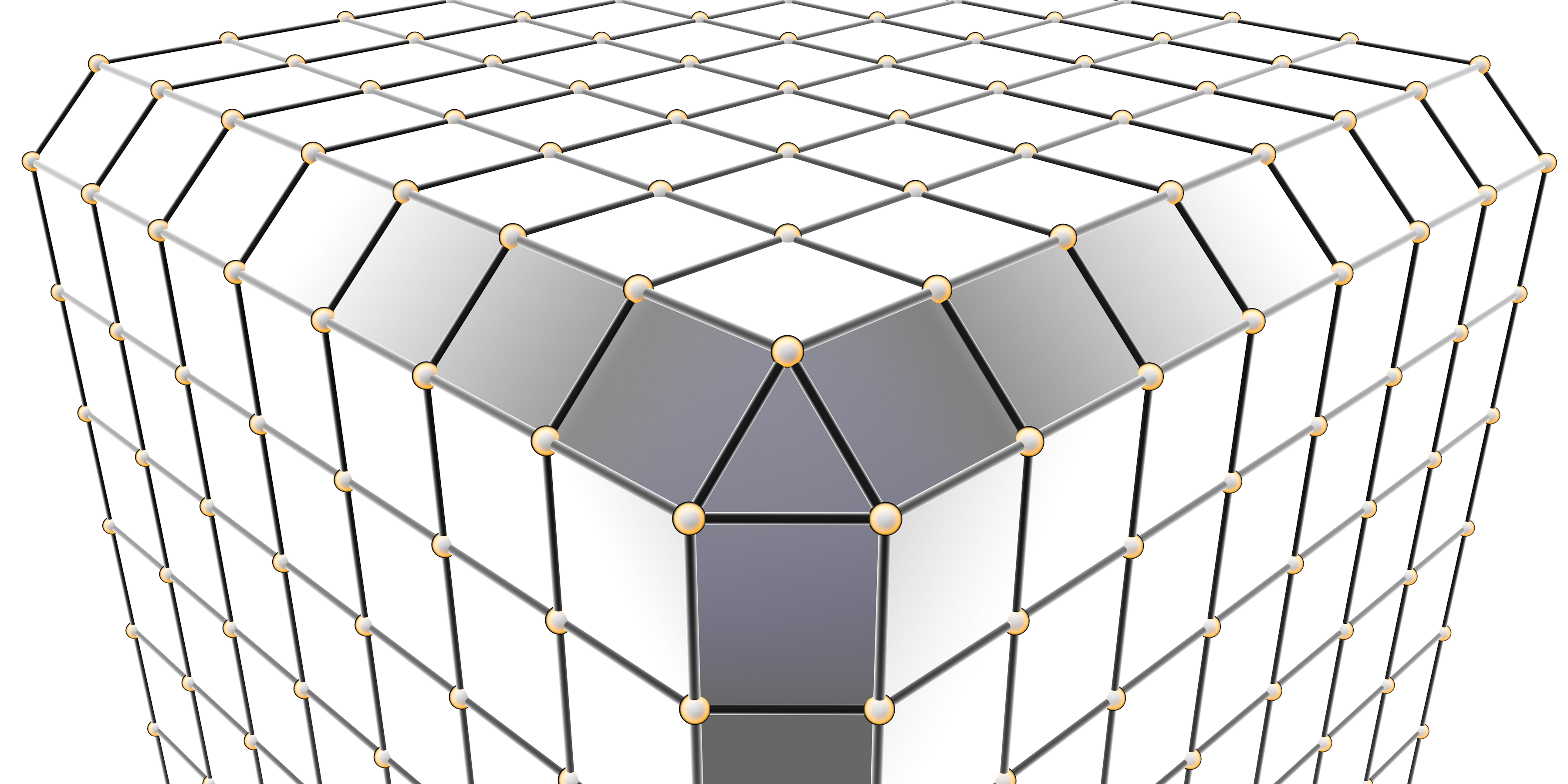}}
    \caption{\small The base complex of the gear system from Figure \ref{fig:gear-corner} has a single triangular 2-cell at the corner where local curvature is focused.}
    \label{fig:gear-corner-frame}
\end{figure}

\subsection{Mobility and conservation}
\label{sec:curvature-conservation}

The curvature $\mu = \coboundary\tilde\coupling$ depends on the extension $\tilde\coupling$ chosen on free edges. Changing $\tilde\coupling(e)$ for a free edge $e$ toggles $\mu$ on the two faces incident to $e$: curvature hops from one face to its neighbor. This mobility respects a conservation law.

If a region $D$ has $\partial D \subset \cgraph$ -- its entire boundary consists of constraint edges -- then the total curvature $\sum_{f \subset D} \mu(f)$ is determined by $\coupling|_{\partial D}$ alone (Lemma~\ref{lem:extension-independence}). Curvature can redistribute among faces of $D$ but cannot cross the constrained boundary. The relative class $[\mu] \in H^2(D, \partial D)$ is an invariant; only its representative -- which specific faces are frustrated -- depends on the extension.

On a closed surface, Stokes constrains the global total:
\[
\sum_{f \in X} \mu(f) = \sum_{e \in \partial X} \coupling(e) = 0.
\]
Frustrated faces pair; isolated defects are topologically forbidden. The spherical zonohedra obey this: their pentagonal rosettes, individually frustrated, pair across the sphere to achieve net neutrality.

The plane $\R^2$, being non-compact, admits no such constraint. A single frustrated region can appear in isolation, its compensating partner pushed arbitrarily far along a chain of free edges. From our perspective this remains a relative phenomenon: the defect is detected on a disc $D$ with $\partial D \subset \cgraph$, as a class in $H^2(D, \partial D)$. Enlarging $D$ simply moves the region under consideration; any compensating defect lies outside.

\section{Absolute $H^2$: Inaccessibility}
\label{sec:absH2}

The relative $H^2$ phenomena of \S\ref{sec:relH2} -- rosette defects, gear corners, 
hyperbolic heptagons -- all arise from the Stokes principle: the connecting 
homomorphism promotes $H^1$ boundary classes to $H^2$ interior curvature. 
These are genuine $H^2$ classes, but they require a bounding cycle as input.

Can bistable systems exhibit \emph{absolute} $H^2$ phenomena -- \style{inaccessibility} 
intrinsic to the constraint structure on a closed surface?
Bistable elements live on vertices; 
pairwise constraints live on edges; the constraint graph $\cgraph$ is 
one-dimensional. 
To access $H^2$ on closed surfaces, one must exploit the 2-cells in the ambient complex $X$. 

Two routes present themselves. The \emph{multiplicative} route uses the cup product: given $H^1$ 
classes on a closed surface, their product $\alpha \smile \beta \in H^2$ measures 
geometric interference -- forced intersections among the seams where torsors 
fail to extend. This requires $H^1$ classes as raw material, which are not always available. In contrast, there is a 
\emph{degree-shifted} route which places bistable variables on edges rather than 
vertices, with prescribed flux on faces; the obstruction to finding a potential 
then lives directly in $H^2$, but the price is redefining what bistable elements are.

The two routes access genuinely different aspects of inaccessibility. The \emph{multiplicative} route via cup products measures interference among existing $H^1$ obstructions -- forced collisions of seams that cannot be separated. The \emph{degree-shifted} route via flux and potentials captures inaccessibility in its purest form: configurations partitioned into sectors by a conserved topological invariant, unreachable from one another despite free local motion. The experiential character differs qualitatively from $H^1$ impossibility: the system is not frozen but \emph{partitioned}.

\subsection{Cup products: seam interference}
\label{sec:cup-product}

The cohomology ring $H^*(X; \Z_2)$ carries a multiplicative structure via the 
cup product $\smile: H^1(X) \times H^1(X) \to H^{2}(X)$. For classes 
$\alpha, \beta \in H^1$ on a closed surface, the product $\alpha \smile \beta \in H^2$ has a geometric interpretation through Poincar\'e duality: 
each $H^1$ class has a dual 1-cycle representable by a closed curve, and 
$\langle \alpha \smile \beta, [X] \rangle$ counts (mod 2) the intersections 
of these dual curves. When the cup product is nonzero, the curves must intersect.

For a $\Z_2$-torsor classified by $\alpha \in H^1(X)$, the Poincar\'e dual 
admits concrete realization as a \style{seam}: a cut along which a chosen local section cannot 
be continued, forcing a sheet swap upon crossing. Away from the seam, 
a consistent section exists; the seam is where local trivializations collide.
Different cocycle representatives move the seam but preserve its homology class.

The torus gear mesh (Figure \ref{fig:gear-torus}) makes this visible. Consider an $N \times M$ grid of 
external gears on $T^2$, with $N$ and $M$ both odd. 
The coupling $\coupling \equiv 1$ enforces opposition at every mesh point, 
and both fundamental cycles carry nonzero holonomy: $[\coupling] = \alpha + \beta$ 
in $H^1(T^2; \Z_2) \cong \Z_2 \oplus \Z_2$. No global spin assignment exists.

The rectangular grid provides a canonical decomposition. Let $\coupling_h \in C^1(T^2; \Z_2)$ be the cochain supported on horizontal edges and $\coupling_v$ the complementary cochain supported on vertical edges. Then $\coupling = \coupling_h + \coupling_v$, and by construction $[\coupling_h] = \alpha$, $[\coupling_v] = \beta$ where $\alpha, \beta$ are the standard generators of $H^1(T^2; \Z_2)$. Let $\Gamma_\alpha$, $\Gamma_\beta$ denote the Poincar\'e-dual seams -- the loci where local sections of the corresponding torsors must swap sheets. The cup product measures whether these seams can be separated:
\[
\langle \alpha \smile \beta, [T^2] \rangle = 
\#(\Gamma_\alpha \cap \Gamma_\beta) \equiv 1 \pmod 2.
\]
When nonzero, no choice of cocycle representatives renders the seams disjoint. On the torus with both dimensions odd, the horizontal and vertical seams must collide -- their intersection is forced by topology, not by the particular decomposition chosen 
(Figure~\ref{fig:torus-seams}). The allowed moves are coboundary shifts $\alpha \mapsto \alpha + \coboundary\xi$, which reposition seams without changing their homology classes; the cup product $\alpha \smile \beta \neq 0$ obstructs accessing a disjoint-seam representative under these moves. Mechanically, this means any attempt to ``resolve'' one obstruction (say, by cutting along a horizontal seam) necessarily crosses the other.

\begin{figure}[htb]
    \centering
    \setlength{\fboxsep}{0pt}
    \includegraphics[height=5cm]{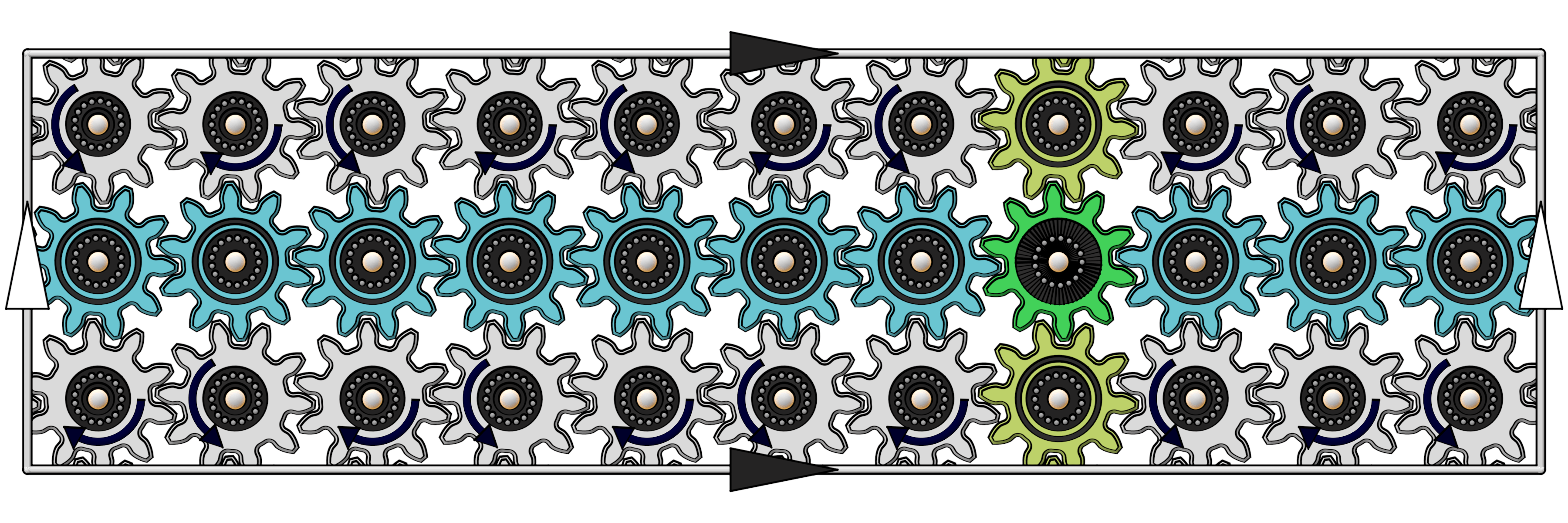}
    \caption{\small A torus gear mesh with both dimensions odd. The horizontal 
    and vertical seams must intersect and represent the Poincar\'e dual of the 
    nonzero cup product of the $H^1$ classes in $H^2$.
}
    \label{fig:torus-seams}
\end{figure}

A single class can interfere with \emph{itself} when the base space is 
nonorientable. On the real projective plane, the cohomology ring is 
$H^*(\mathbb{RP}^2; \Z_2) = \Z_2[\alpha]/(\alpha^3)$, and the cup square 
$\alpha^2 \neq 0$ generates $H^2$. To access $\mathbb{RP}^2$ from gears, 
consider a bipartite mesh on $S^2$ -- freely spinning, no obstruction -- 
but viewed through antipodal apertures subject to the opposite-spin rule 
of \S\ref{sec:dual-apertures}. The configuration space of antipodal pairs 
is $\mathbb{RP}^2$, and the opposite-spin rule constructs a torsor classified 
by $\alpha \in H^1(\mathbb{RP}^2)$. Its seam is a one-sided thickened equator with unavoidable self-intersection (Figure~\ref{fig:rp2-gears}).

\begin{figure}[htb]
    \centering
    \setlength{\fboxsep}{0pt}
    \includegraphics[width=6in]{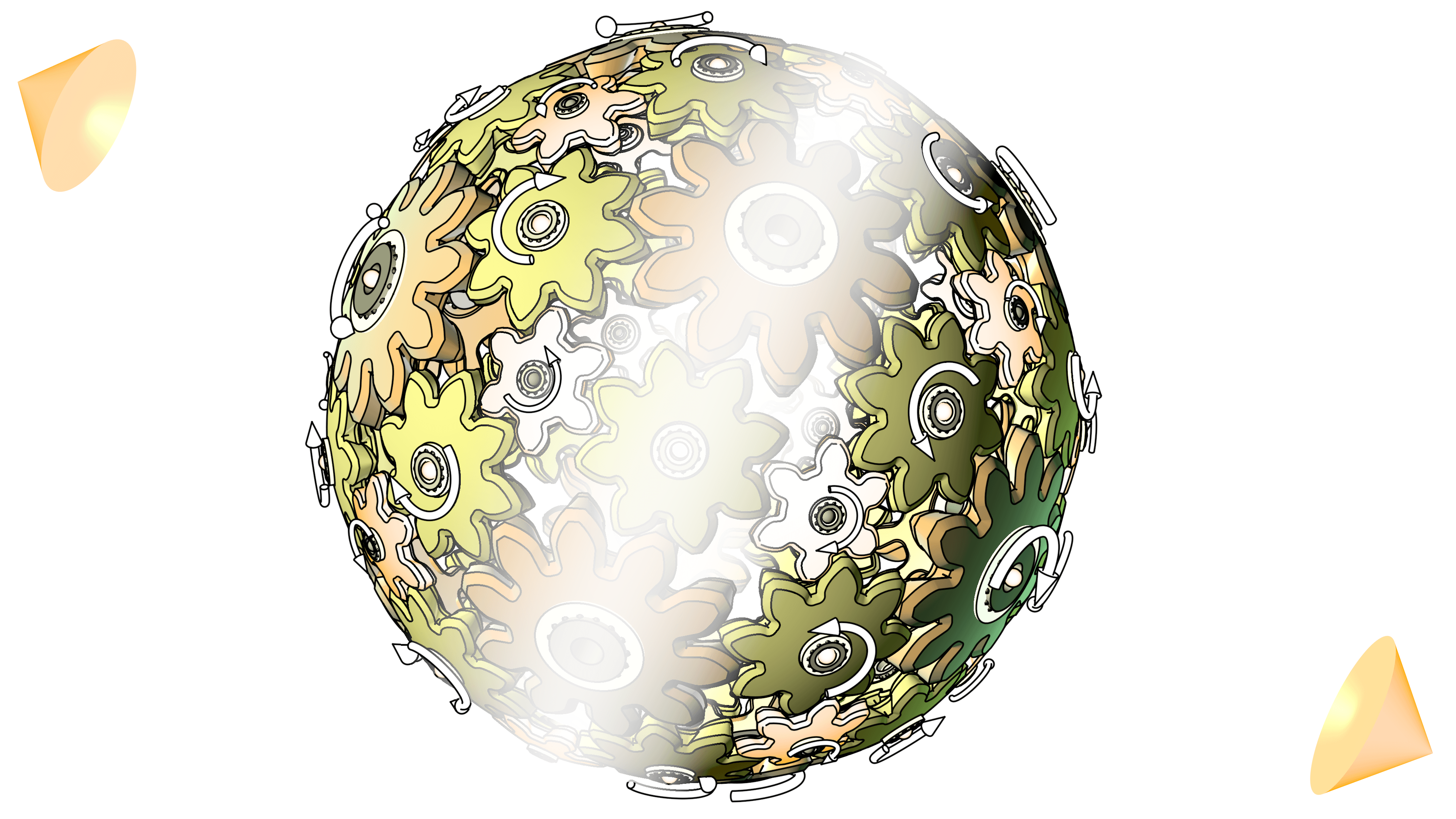}
    \caption{\small A pair of lights illuminates antipodal apertures of a gear mesh on $S^2$ 
    with the opposite-spin rule. The configuration space is $\mathbb{RP}^2$; the constructed torsor 
    has cohomology class $\alpha\in H^1(\mathbb{RP}^2)$ with $\alpha^2 \neq 0$.
}
    \label{fig:rp2-gears}
\end{figure}

\subsection{Degree-shifted models: flux and potentials}
\label{sec:flux}

The cup product route requires $H^1$ classes as raw material -- it measures 
interference among existing obstructions. To access $H^2$ without this 
prerequisite, we shift the degree at which bistable variables live.

Throughout the preceding sections, bistable states were 0-cochains 
$x \in C^0(X; \Z_2)$ assigning values to vertices, with coupling 
$\coupling \in C^1$ specifying edge constraints. The obstruction to 
solving $\coboundary x = \coupling$ lived in $H^1$. Now shift everything 
up by one degree: bistable states become 1-cochains on edges, constraints 
become 2-cochains on faces, and the obstruction moves to $H^2$.

Let $X$ be a closed surface with a fixed cellulation. A \style{flux} is a 
2-cochain $\mu \in C^2(X; \Z_2)$ assigning a parity to each face. A 
\style{potential} for $\mu$ is a 1-cochain $A \in C^1(X; \Z_2)$ satisfying 
$\coboundary A = \mu$ -- the edge-states sum to the prescribed face-value 
around each boundary. A flux admits a potential if and only if 
$[\mu] = 0$ in $H^2(X; \Z_2)$.

On a closed surface, Stokes constrains total flux: for any potential $A$,
\[
\sum_{f \in X} (\coboundary A)(f) = \langle \coboundary A, [X] \rangle = 0.
\]
Every exact flux has even total parity. The evaluation map 
$[\mu] \mapsto \sum_f \mu(f)$ identifies $H^2(X; \Z_2) \cong \Z_2$ for 
any connected closed surface: odd total parity represents the nontrivial 
class and admits no potential.

The obstruction manifests operationally as \style{inaccessibility}. Define 
an \style{edge toggle} as the operation $A \mapsto A + \epsilon_e$ flipping 
the potential on edge $e$; this changes flux by $\mu \mapsto \mu + \coboundary\epsilon_e$, 
toggling exactly the two incident faces. Starting from any configuration, 
the reachable fluxes under edge toggles form the coset $\mu + B^2(X)$. 
Configurations partition into $|H^2(X)|$ sectors -- for a closed surface, 
two: even-parity and odd-parity fluxes, mutually unreachable.

This is a different species of obstruction than $H^1$ impossibility. An odd 
gear cycle cannot move; contradiction is immediate. The $H^2$ obstruction 
here is kinematic: every edge toggle executes freely, the system responds 
fluidly to local moves, yet certain configurations remain forever out of 
reach. The system is not locked but \emph{partitioned} -- the hallmark of inaccessibility.

\begin{figure}[htb]
    \centering
    \setlength{\fboxsep}{0pt}
    \fbox{\includegraphics[width=6.0in]{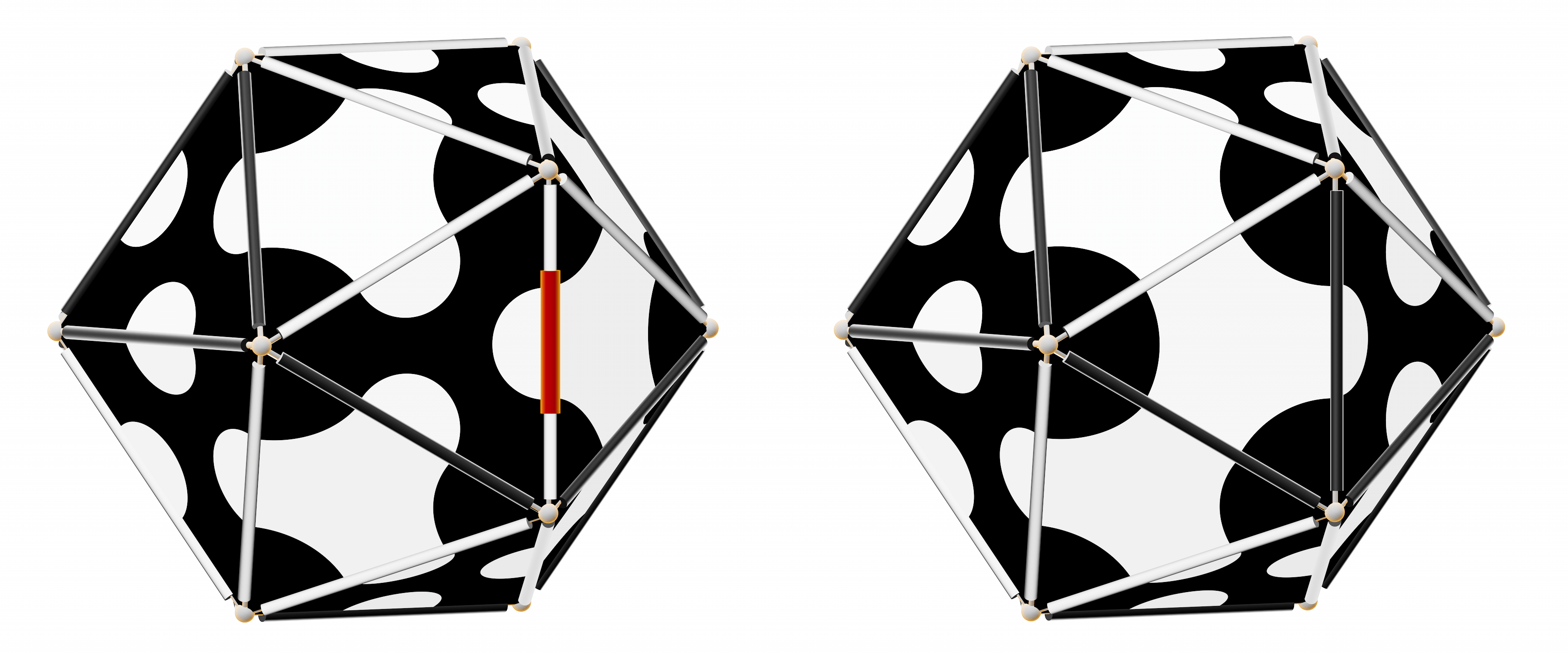}}
    \caption{\small The truchet game on a triangulated sphere. Clicking an edge 
    [left, in red] toggles its state and toggles the states of the two adjacent faces [right]; 
    the goal is to reach a target configuration of face tiles. Victory is possible if and only if 
    the initial and target face distributions share the same $H^2$ class.
    \animationlink[Interactive]{https://github.com/robertghrist/truchet-icosahedron/}}
    \label{fig:truchet-game}
\end{figure}

This inaccessibility can be experienced directly through a simple game. 
Tile a closed surface with truchet-style triangles admitting two states each: see 
Figure~\ref{fig:truchet-game}. Clicking an edge toggles both 
incident faces. The challenge: reach a target configuration from a 
random start. Every move executes freely, yet on a closed surface parity is conserved -- some targets 
are unreachable regardless of the sequence of moves. The $H^2$ class 
partitions configurations into exactly two sectors, and no sequence of local moves can cross from one to the other. 
This is a simple $H^2$ reachability problem.

This game echoes the classic \emph{Lights Out} puzzle, where pressing a button toggles neighboring lights and the goal is to extinguish all \cite{AndersonFeil1998LightsOut}. Both reduce to $\Z_2$ linear algebra; both exhibit ``chase the defect'' dynamics where local corrections create new problems elsewhere. In our setting, the algebra is cellular cohomology, and the conserved quantity partitioning winnable from unwinnable configurations is the $H^2$ class.

One can play on any triangulated or cellular surface. 
For surfaces with boundary, accessibility depends on whether boundary edges can be toggled. With \style{free boundary} (all edges clickable), every configuration is reachable: defects can be pushed to the boundary and annihilated. With \style{frozen boundary}, the obstruction migrates to relative cohomology $H^2(D, \partial D) \cong \Z_2$; parity relative to the fixed boundary frame is conserved.

\section{Conclusions}
\label{sec:conc}

Penrose's enigmatic heptagonal stairs exploited bistable ambiguities around a loop and pointed to $\Z_2$ coefficients and parity arithmetic -- the simple tools here elevated to the language of torsors and cohomological classification. The hierarchy that we have traced is echoed in Necker cubes, gears, tilings, and more. What cohomology detects at each degree corresponds to a qualitatively distinct mode of visual paradox, and the discrete Stokes theorem -- boundary holonomy inducing interior curvature -- is the single mechanism that drives each transition.

Along the way, the framework has clarified several phenomena that were previously understood only in isolation, from the classical intuitions about gear rings and parity to the M\"obius gear ring, and the P3 quasitilings, the latter exhibiting a relative $H^2$ defect -- boundary holonomy forcing interior curvature. The cup product on the torus gear mesh reveals that $H^2$ inaccessibility can arise as a consequence of independent $H^1$ obstructions on a closed surface.

Several natural extensions press against the boundaries of this framework. The $\Z_2$ coefficient restriction invites relaxation: gear systems with specified ratios demand rational or real coefficients; full stepped-surface heights require integer coefficients; bevel gears coupling rotation axes at angles carry holonomy valued in $SO(3)$. The P3 quasitiling, with its heterogeneous mix of even and odd cycles, calls for spatially-varying constraint data -- the language of sheaves rather than uniform cochains. The prior work \cite{GhristCooperband2025Obstructions} developed sheaf-theoretic methods for static paradoxes and the general theory of nonabelian network torsors; combining those tools with the hierarchical structure developed here is a natural next step.

The cohomological hierarchy itself continues. Just as $H^1$ classifies torsors and $H^2$ classifies gerbes, $H^3$ classifies 2-gerbes, and the pattern persists -- though physical or visual realizations of higher cohomological obstructions remain to be discovered. Perhaps the most promising direction is temporal. The Method of Monodromic Apertures uses time to reveal spatial structure, but treats the animation as a parameter rather than part of the base. A theory of paradoxes intrinsic to \emph{video} -- where the impossibility emerges from the sequence of frames rather than residing in any single one -- would require cohomology of a space-time complex. The machinery is ready; the visual imagination has not yet caught up.

\bibliographystyle{unsrturl}

\bibliography{IMPOSSIBLE}

\appendix

\section{Cohomological Background}
\label{sec:appendix-cohomology}

This appendix collects the cohomological tools used in the paper. We work 
exclusively with $\Z_2 = \Z/2\Z$ coefficients, which eliminates all sign 
considerations. Readers seeking a fuller treatment should consult 
Hatcher~\cite{Hatcher2002} or any standard text; our aim here is to fix 
notation and recall the specific facts required.

\subsection{Cochains, coboundary, and cohomology}

Let $X$ be a cell complex. The \style{$k$-cochain group} is
\[
C^k(X) = C^k(X; \Z_2) = \{ f: \{\text{$k$-cells of $X$}\} \to \Z_2 \},
\]
a $\Z_2$-vector space with basis dual to the $k$-cells. The 
\style{coboundary operator} $\coboundary: C^k(X) \to C^{k+1}(X)$ is defined by
\[
(\coboundary c)(e) = \sum_{f \subset \partial e} c(f),
\]
summing over $k$-cells in the boundary of the $(k+1)$-cell $e$. For a graph, 
$(\coboundary c)(e) = c(v) + c(u)$ when $e = uv$. For a $2$-complex, 
$(\coboundary c)(f) = \sum_{e \subset \partial f} c(e)$.

The identity $\coboundary^2 = 0$ holds because each $(k-1)$-cell appears an 
even number of times in $\partial(\partial e)$. This yields the 
\style{cohomology groups}
\[
H^k(X) = H^k(X; \Z_2) = \frac{\ker(\coboundary: C^k \to C^{k+1})}
{\mathrm{im}(\coboundary: C^{k-1} \to C^k)} = \frac{Z^k(X)}{B^k(X)}.
\]
For graphs, the key characterization: a $1$-cochain $\coupling$ is a 
coboundary if and only if $\sum_{e \in \gamma} \coupling(e) = 0$ for every 
cycle $\gamma$.

\subsection{Relative cohomology and the long exact sequence}

For a subcomplex $A \subset X$, the \style{relative cochain groups} are
\[
C^k(X, A) = \{ c \in C^k(X) : c|_A = 0 \},
\]
with cohomology $H^k(X, A)$ defined analogously. These fit into a 
\style{long exact sequence}
\[
\cdots \to H^k(X) \xrightarrow{\mathrm{res}} H^k(A) 
\xrightarrow{\connecting} H^{k+1}(X,A) \to H^{k+1}(X) \to \cdots
\]
The \style{connecting homomorphism} $\connecting$ is computed as follows: 
given a cocycle $a$ on $A$, extend it arbitrarily to a cochain $\tilde{a}$ 
on $X$, then $\connecting([a]) = [\coboundary \tilde{a}]$. This is 
well-defined and measures the obstruction to extending $a$ inward: 
$\connecting([a]) = 0$ if and only if $a$ extends to a cocycle on all of $X$.

When $X$ is contractible, exactness forces an isomorphism 
$\connecting: H^k(A) \xrightarrow{\cong} H^{k+1}(X, A)$.

\subsection{$\Z_2$-torsors and double covers}

A \style{$\Z_2$-torsor} over $X$ is a two-sheeted covering space with no 
preferred sheet: a space $\torsor$ with a free $\Z_2$-action and projection 
$\pi: \torsor \to X$ whose fibers are $\Z_2$-orbits. A \style{global section} 
is a continuous choice of one sheet over every point; a torsor is 
\style{trivial} if and only if it admits one.

Over $S^1$, there are exactly two torsors: the trivial one (two disjoint 
circles) and the connected double cover (the boundary of a M\"obius band). 
The connected cover admits no global section.

A coupling $\coupling \in C^1(\cgraph; \Z_2)$ on a graph $\cgraph$ 
determines a torsor by the following construction. Take two copies of each 
vertex. For each edge $e = uv$: if $\coupling(e) = 0$, connect 
$0_u \leftrightarrow 0_v$ and $1_u \leftrightarrow 1_v$; if $\coupling(e) = 1$, 
connect $0_u \leftrightarrow 1_v$ and $1_u \leftrightarrow 0_v$. The 
resulting graph is the total space $\torsor$, connected if and only if 
$[\coupling] \neq 0$ in $H^1(\cgraph)$.

\begin{theorem}[Classification of $\Z_2$-torsors]
\label{thm:torsor-classification}
For a connected complex $X$, there is a natural bijection
\[
\{ \text{isomorphism classes of $\Z_2$-torsors over } X \} 
\;\longleftrightarrow\; H^1(X; \Z_2).
\]
The trivial torsor corresponds to zero; nontrivial classes correspond to 
connected double covers.
\end{theorem}

The double-cover construction of \S\ref{sec:torsors} sends a cocycle 
$\coupling$ to a torsor $\torsor_\coupling$. Cohomologous cocycles yield 
isomorphic torsors: if $\coupling' = \coupling + \coboundary\xi$, then 
relabeling sheets by $\xi$ provides the isomorphism. Conversely, any double 
cover determines a cocycle by recording, for each edge, whether the two 
sheets connect straight or crossed. The correspondence respects the group 
structure: tensor product of torsors corresponds to addition of cocycles.
For more about network torsors and cohomology, see \cite{GhristCooperband2025Obstructions}.

\subsection{Cup products and Poincar\'e duality}

The \style{cup product} $\cupprod \,\colon H^p(X) \times H^q(X) \to H^{p+q}(X)$ 
makes $H^*(X; \Z_2)$ into a graded ring, associative and commutative 
(over $\Z_2$, signs are trivial). On closed surfaces, the cup product in 
$H^1 \times H^1 \to H^2$ has a geometric interpretation via 
\style{Poincar\'e duality}.

For a closed $n$-manifold $M$, Poincar\'e duality gives an isomorphism 
$\mathrm{PD}: H^k(M) \cong H_{n-k}(M)$. A class $\alpha \in H^1$ of a 
surface corresponds to a homology class representable by a closed curve 
$\gamma_\alpha$. The cup product translates to intersection:
\[
\mathrm{PD}(\alpha \cupprod \beta) = \gamma_\alpha \cap \gamma_\beta.
\]
In particular, $\alpha \cupprod \beta \neq 0$ if and only if the 
representing curves cannot be isotoped apart.

\subsection{$\Z_2$-gerbes}

For our purposes, a $\Z_2$-gerbe on $X$ is simply a class $[\mu] \in H^2(X; \Z_2)$. 
Just as torsors classify the ambiguity in choosing sections, gerbes classify the 
ambiguity in choosing trivializations. A 
\style{trivialization} is a $1$-cochain $A$ with $\coboundary A = \mu$; 
such an $A$ exists if and only if $[\mu] = 0$. When trivializations exist, 
they form a torsor over $H^1(X)$: the ambiguity in choosing a potential 
mirrors the ambiguity in choosing a global section one degree down.

\end{document}